\documentclass[11pt,letterpaper,reqno]{amsart}
\usepackage[margin=1.25in]{geometry}
\usepackage[T1]{fontenc}
\usepackage[utf8]{inputenc}
\usepackage{lmodern}
\usepackage[expansion=false]{microtype}
\usepackage{amsmath,amssymb,amsthm,mathtools,mathrsfs}
\usepackage{enumitem,needspace}
\usepackage[hidelinks]{hyperref}
\numberwithin{equation}{section}
\allowdisplaybreaks[2]
\setlist[enumerate]{leftmargin=2.3em,label=\textup{(\roman*)}}
\newtheorem{theorem}{Theorem}[section]
\newtheorem{proposition}[theorem]{Proposition}
\newtheorem{lemma}[theorem]{Lemma}
\newtheorem{corollary}[theorem]{Corollary}
\theoremstyle{remark}
\newtheorem{remark}[theorem]{Remark}
\newcommand{\D}{\mathbb D}
\newcommand{\C}{\mathbb C}
\newcommand{\R}{\mathbb R}

\newcommand{\Z}{\mathbb Z}
\newcommand{\T}{\mathbb T}
\newcommand{\E}{\mathbb E}
\newcommand{\Prob}{\mathbb P}
\newcommand{\Hol}{\mathcal H(\D)}
\newcommand{\Conf}{\mathcal N(\D)}
\newcommand{\Law}{\mathcal L}
\newcommand{\ind}{\mathbf 1}
\newcommand{\X}{\mathfrak X}
\newcommand{\norm}[1]{\lVert#1\rVert}
\DeclareMathOperator{\Var}{Var}

\DeclareMathOperator{\Aut}{Aut}
\DeclareMathOperator{\Log}{Log}
\DeclareMathOperator{\Arg}{Arg}
\DeclareMathOperator{\supp}{supp}
\DeclareMathOperator{\Haar}{Haar}
\DeclareMathOperator{\Thin}{Thin}
\newcommand{\dm}{\,dm}
\newcommand{\ip}[2]{\langle #1,#2\rangle}
\hypersetup{
 pdftitle={Canonical analytic realizations of hyperbolic determinantal processes},
 pdfauthor={Xiang Fang, Feng Guo, Shengzhao Hou, Qi Zhou},
 pdfsubject={Random analytic functions, intrinsic reconstruction, and Moebius covariance},
 pdfkeywords={Bergman determinantal point process, random analytic function, Blaschke product, Barnes G-function}
}
\title[Canonical analytic realizations]
{Canonical analytic realizations of hyperbolic determinantal processes}
\author{Xiang Fang}
\address{
Department of Applied Mathematics,
National Yang Ming Chiao Tung University,
Hsinchu, Taiwan
}
\email{xfang@nycu.edu.tw}

\author{Feng Guo}
\address{
School of Mathematics, South China University of Technology, Guangzhou 510640, P. R. China
}
\email{70207994@nuaa.edu.cn}

\author{Shengzhao Hou}
\address{
School of Mathematical Sciences,
Soochow University,
Suzhou 215006, P. R. China
}
\email{shou@suda.edu.cn}

\author{Qi Zhou}
\address{
School of Mathematical Sciences,
Soochow University,
Suzhou 215006, P. R. China
}
\email{zhouqi@suda.edu.cn}
\date{}

\begin{document}
\begin{abstract}
Krishnapur asked whether the invariant hyperbolic determinantal point
processes on the disk admit a random analytic zero-set interpretation at
noninteger parameters. We construct such a realization for every positive
real parameter as the full compact-open limit in distribution of normalized
finite Blaschke products. The zeros determine the modulus and normalized
analytic shape, leaving one independent uniform phase. We prove exact
M\"obius covariance and classify all realizations with this covariance and
square-integrable logarithmic modulus at the origin: they are precisely
independent positive random multiples of the canonical function. Within this covariant class, matching the canonical logarithmic mean and
variance uniquely determines the canonical function law.  The family
is weakly continuous in the parameter and agrees at positive integers with
determinants of matrix-valued Gaussian power series. An explicit Barnes
$G$-function Mellin transform determines the basepoint normalization.
\end{abstract}
\maketitle

\section{Introduction}\label{sec:introduction}

Krishnapur asked whether the invariant hyperbolic determinantal point
processes on the disk admit a random analytic zero-set interpretation at
noninteger parameters~\cite[Section~8, Question~1]{Krishnapur2009}.
These processes form a family $H_\alpha$, $\alpha>0$, whose parameter is
the intensity relative to hyperbolic area. 
At $\alpha=1$, Peres and
Vir\'ag~\cite{PV2005} identified $H_1$ as the zeros of the power series
with independent standard complex Gaussian coefficients. Krishnapur
realized $H_m$ for each positive integer $m$ by the determinant of an
$m\times m$ matrix-valued Gaussian power series, and
also established the corresponding function-level limit from truncated
Haar unitary matrices~\cite{Krishnapur2009}. The determinantal law
exists for every positive real parameter, whereas the matrix realization
uses that parameter as its dimension. This gap is also discussed
in~\cite[Section~4.3.11]{HKPV2009}.

We answer Krishnapur's question for every $\alpha>0$. The realizing random
analytic function $F_\alpha$ is obtained as the full compact-open limit
in distribution of normalized finite Blaschke products. The laws
$\Law(F_\alpha)$ vary continuously with $\alpha$ and agree, at the positive
integers, with the complete function laws of the Gaussian determinants.
Thus the construction extends the integer models at the level of random
analytic functions as well as their zero processes.

The function law is a substantial part of the conclusion. Prescribing an
analytic zero divisor leaves an arbitrary zero-free analytic factor.
Here the zeros determine both the modulus and the normalized analytic
shape; conditional on the zeros, the only randomness is one independent
uniform phase. Moreover, within the class with the prescribed M\"obius
covariance and square-integrable logarithmic modulus at the origin, every
realization of $H_\alpha$ is an independent positive random multiple of
$F_\alpha$. Two logarithmic moments select the canonical law. This
characterization explains the normalization without prescribing an entire
basepoint distribution or a conditional phase law.

The family itself has an invariant characterization. Krishnapur classified
the invariant holomorphic projection kernels with a radial reference
measure~\cite[Chapter~3, Theorem~3.0.5]{Krishnapur2006}. In
Appendix~\ref{sec:invariant-kernels} we give the corresponding statement
for an arbitrary positive Radon reference measure. Allowing positive
contractions gives precisely independent thinnings of the processes
$H_\alpha$; the projection case is exactly the unthinned family. This
formulation extends Krishnapur's kernel argument and specifies the class
realized in this paper. Our analytic realization theorems concern this
projection family.

\subsection{The model and three main results}\label{subsec:main-results}

Let $\D=\{z\in\C:|z|<1\}$, let $dm$ be planar Lebesgue measure, and put
$D_r=\{z:|z|<r\}$. Equip the space $\Hol$ of analytic functions on $\D$
with the compact-open topology and the space $\Conf$ of locally finite
integer-valued Radon measures with the vague topology. For
$f\not\equiv0$, the zero divisor $Z(f)\in\Conf$ counts multiplicities.

For $\alpha>0$, write
\begin{equation}\label{eq:model}
 d\mu_\alpha(z)=\frac{\alpha}{\pi}(1-|z|^2)^{\alpha-1}dm(z),
 \qquad K_\alpha(z,w)=(1-z\bar w)^{-\alpha-1}.
\end{equation}
The kernel uses the analytic branch equal to $1$ at $z\bar w=0$.
It projects $L^2(\mu_\alpha)$ onto its analytic subspace, and $H_\alpha$
denotes its determinantal law. Its intensity is
\begin{equation}\label{eq:intensity}
 d\nu_\alpha(z)=\rho_\alpha(z)dm(z),\qquad
 \rho_\alpha(z)=\frac{\alpha}{\pi(1-|z|^2)^2}.
\end{equation}
Let $\Xi_{N,\alpha}$ be the rank-$N$ projection process with kernel
\begin{equation}\label{eq:finite-kernel}
 K_{N,\alpha}(z,w)=\sum_{n=0}^{N-1}\binom{n+\alpha}{n}(z\bar w)^n
 \quad\text{relative to }\mu_\alpha.
\end{equation}
It has exactly $N$ points almost surely. Define
\begin{equation}\label{eq:finite-function}
 F_{N,\alpha}(z)=N^{\alpha/2}
       \prod_{x\in\Xi_{N,\alpha}}\frac{z-x}{1-z\bar x}.
\end{equation}
The phase is the one prescribed by this finite product.

Let $G$ be the Barnes function, normalized by $G(1)=1$ and
$G(z+1)=\Gamma(z)G(z)$, and set
\begin{equation}\label{eq:barnes-function}
 M_\alpha(u)=\frac{G(1+\alpha+u)}{G(1+\alpha)G(1+u)},
 \quad \operatorname{Re}u>-1,\qquad
 c_\alpha=\frac12(\log M_\alpha)'(0).
\end{equation}
Complex powers of positive random variables use their real logarithm.

\begin{theorem}[Canonical realization and reconstruction]
\label{thm:realization}\label{thm:reconstruction}
For every $\alpha>0$, there is a random $F_\alpha\in\Hol$ with the
following properties.
\begin{enumerate}
\item Along the full sequence,
\begin{equation}\label{eq:full-limit}
 (F_{N,\alpha},\Xi_{N,\alpha})\Longrightarrow
 (F_\alpha,Z(F_\alpha))\quad\text{in }\Hol\times\Conf.
\end{equation}
Moreover, $F_\alpha(0)\ne0$ almost surely and $Z(F_\alpha)\sim H_\alpha$.
\item There are Borel maps $T_0:\Conf\to\R$, $P:\Conf\to\Hol$ and an
$H_\alpha$-full Borel set $\mathscr G_\alpha$ such that
$P_\xi(0)=1$ and $Z(P_\xi)=\xi$ on that set. On the coupling
$\Xi=Z(F_\alpha)$,
\begin{equation}\label{eq:intrinsic-identities}
 |F_\alpha(0)|=e^{c_\alpha-T_0(\Xi)},\qquad
 \frac{F_\alpha}{F_\alpha(0)}=P_\Xi
 \quad\text{almost surely}.
\end{equation}
The phase $U=F_\alpha(0)/|F_\alpha(0)|$ is Haar on $\T$ and independent
of $\Xi$. Conversely, for independent $\Xi\sim H_\alpha$ and
$U\sim\Haar(\T)$,
\begin{equation}\label{eq:sampling}
 Ue^{c_\alpha-T_0(\Xi)}P_\Xi\ \overset d=\ F_\alpha.
\end{equation}
\item The basepoint modulus satisfies
\begin{equation}\label{eq:main-barnes}
 \E|F_\alpha(0)|^{2u}=M_\alpha(u),\qquad\operatorname{Re}u>-1.
\end{equation}
This half-plane is maximal for absolute Mellin moments.
\end{enumerate}
\end{theorem}

The maps $T_0$ and $P$ admit the fixed radial representations in
Proposition~\ref{prop:canonical}. Thus~\eqref{eq:intrinsic-identities}
identifies measurable functions of the actual zero configuration in the
joint limit. The independence of the remaining phase is a conclusion
about the finite-product limit.

To state the characterization, write, for $\phi\in\Aut(\D)$,
\begin{equation}\label{eq:automorphism}
 \phi(z)=e^{i\vartheta}\frac{a-z}{1-\bar az},\qquad a=\phi^{-1}(0),
\end{equation}
and define the positive-at-zero multiplier
\begin{equation}\label{eq:multiplier}
 C_{\alpha,\phi}(z)=(1-|a|^2)^{-\alpha/2}
                 \exp\{\alpha\Log(1-\bar az)\},
\end{equation}
where $\Log(1-\bar az)$ vanishes at zero. Put
\begin{equation}\label{eq:log-variance-normalization}
 v_\alpha=\frac14(\log M_\alpha)''(0)
          =\Var(\log|F_\alpha(0)|).
\end{equation}

\begin{theorem}[M\"obius covariance and classification of realizations]
\label{thm:covariance}\label{thm:geometric}
Fix $\alpha>0$.
\begin{enumerate}
\item For every $\phi\in\Aut(\D)$,
\begin{equation}\label{eq:main-covariance}
 F_\alpha\circ\phi\ \overset d=\ C_{\alpha,\phi}F_\alpha
 \quad\text{in }\Hol.
\end{equation}
\item Let $F\in\Hol$ be random, with $F(0)\ne0$ almost surely,
$Z(F)\sim H_\alpha$, and $\E(\log|F(0)|)^2<\infty$. 
Then
$F\circ\phi\overset d=C_{\alpha,\phi}F$ for every
$\phi\in\Aut(\D)$ if and only if
\begin{equation}\label{eq:independent-scale-classification}
 F\overset d=S F_\alpha,
\end{equation}
where $S>0$ is independent of $F_\alpha$.
The law of $S$ is uniquely determined by the law of $F$.
\item Among the covariant realizations in \textup{(ii)}, $F_\alpha$ is
the unique law satisfying
\begin{equation}\label{eq:geometric-normalization}
 \E\log|F(0)|=c_\alpha,\qquad
 \Var(\log|F(0)|)=v_\alpha.
\end{equation}
\end{enumerate}
\end{theorem}

In particular, every realization in part~(ii) has logarithmic variance
at least $v_\alpha$, with equality precisely for deterministic positive
multiples of $F_\alpha$. No phase assumption is needed: the composition
law of the multipliers forces circular symmetry. The proof identifies
the independent scale and Haar phase jointly with the actual zero
configuration.

For reference, covariance gives the point laws and second moments
\begin{align}
 (1-|b|^2)^{\alpha/2}F_\alpha(b)&\overset d=F_\alpha(0),
       &&b\in\D,\label{eq:main-point-law}\\
 \E[F_\alpha(z)\overline{F_\alpha(w)}]
       &=\Gamma(1+\alpha)(1-z\bar w)^{-\alpha},
       &&z,w\in\D.\label{eq:main-covariance-kernel}
\end{align}
These are proved in Corollary~\ref{cor:point-laws}. The second-moment
kernel alone does not determine the function law.

\Needspace{10\baselineskip}
\begin{theorem}[Continuity and the integer models]
\label{thm:continuity}\label{thm:integer}
\begin{enumerate}
\item If $\alpha_j\to\alpha>0$, then
$F_{\alpha_j}\Longrightarrow F_\alpha$ in $\Hol$.
\item Let $m\ge1$ be an integer, and let $G_0,G_1,\ldots$ be independent
$m\times m$ matrices with independent standard complex Gaussian entries,
each entry having density $\pi^{-1}e^{-|z|^2}$. Then
\begin{equation}\label{eq:main-integer}
 F_m\ \overset d=\ \det\Bigl(\sum_{n=0}^\infty G_nz^n\Bigr)
 \quad\text{in }\Hol.
\end{equation}
\end{enumerate}
\end{theorem}

Part~(i) asserts weak continuity of the function laws. Part~(ii)
records consistency with the known integer models. Krishnapur's proof
of~\cite[Theorem~4]{Krishnapur2009} already establishes the
corresponding compact-open function limit at positive integer
parameters. Our proof below gives an alternative identification through
Theorem~\ref{thm:geometric}(iii). At $m=1$ this is the Gaussian power series of
Peres and Vir\'ag.

\subsection{Proof mechanism and relation to earlier work}
\label{subsec:outline}\label{subsec:related}

The main issue in passing from finite configurations to functions is
control of the zero-free factor contributed by points approaching the
boundary. Vague convergence of the configurations does not control this
factor. Our construction starts with the bounded statistic associated
with $1-\left|(z-w)/(1-z\bar w)\right|^2$. Its variance is at most
$\alpha$, and radial beta identities control its mean. The subharmonic
mean inequality then gives compact-open tightness of the finite functions for every $\alpha>0$.

We identify a joint cluster limit through its logarithmic modulus at the
origin and its logarithmic jets. The centered symbol $-\log|w|$ is in
$L^2(\nu_\alpha)$. For a jet, the symbol $\bar w^k-w^{-k}$ has an
$L^2$ boundary tail; its singular interior contribution is controlled by
the probability of a point near zero. Common cutoffs retain the limiting
function and divisor together. The Taylor recursion consequently makes
the normalized analytic shape a measurable function of that divisor.

The prescribed finite phase requires a separate argument. Rotating the
finite configuration through an angle $\theta/N$ gives
\[
 (F_{N,\alpha},\Xi_{N,\alpha})\overset d=
 \bigl(e^{i\theta}F_{N,\alpha}(e^{-i\theta/N}\,\cdot),
                         e^{i\theta/N}\Xi_{N,\alpha}\bigr).
\]
The changes in the argument and configuration disappear in the limit,
while the phase $e^{i\theta}$ remains. Haar averaging identifies the
conditional phase and hence every joint cluster law. This proves
full-sequence convergence. The scalar normalization is computed from the
independent beta radii, using the radial decomposition
in~\cite[Theorem~26]{HKPV2006}.

The geometric characterization compares a candidate with the canonical
function on the same zeros. The logarithmic modulus of their quotient is
a stationary square-integrable harmonic field, hence a spatial constant.
Mixing of \(H_\alpha\) makes this constant independent of the zeros. The
projective multiplier identity forces the remaining conditional phase to
be Haar. Logarithmic mean and variance then remove exactly the independent
scale. Uniform Green estimates give parameter continuity, and the
characterization identifies the integer models.

Several antecedents are particularly close. Peres and
Vir\'ag~\cite[Theorem~6]{PV2005} showed that the hyperbolic Gaussian
analytic function is determined by its zeros up to a unimodular constant.
Their starting point is a given Gaussian analytic function, whereas here
we start from the weighted Bergman determinantal law and ask whether it
can be realized as the zero set of a natural random analytic function for
arbitrary \(\alpha>0\).

A complementary coefficient-level reconstruction in the planar GAF
setting was developed by Ghosh and
Peres~\cite[Theorem~1.5]{GhoshPeres2017}. There, regularized inverse
power sums of the zeros, combined with Newton identities, recover
normalized Taylor coefficients, again leaving an independent uniform
phase. Our logarithmic jets play an analogous role for the weighted
Bergman determinantal family, with the additional requirement that they
arise as limits of the prescribed finite Blaschke products.

Bufetov and Qiu~\cite[Theorem~1.4 and Eq.~(5)]{BQ2017} constructed
regularized multiplicative functionals of Blaschke products for weighted
Bergman processes and used them to describe Palm measures. For the
weighted Bergman process considered here, their single-center logarithmic
statistic is exactly \(-T_z\), with \(T_z\) defined in
Section~\ref{subsec:green-estimates}. This statistic and the regularized
change-of-measure framework are direct antecedents. The proof of
Theorem~\ref{thm:realization} proceeds directly from the finite ensembles
and the elementary determinantal variance bound; it does not invoke their
multiplicative-functional theorem. The additional analytic conclusions
are compact-open convergence of the prescribed finite products,
reconstruction of the normalized shape, and identification of the
normalization and conditional phase. The all-center Green field is
established after this construction and used for the geometric results.

Section~\ref{sec:transforms} gives a complementary characterization by
joint Mellin--Fourier transforms of the basepoint and logarithmic
derivatives. The tilted determinantal laws are defined independently of
the function. Their change of measure is a specialization of the
regularized multiplicative-functional framework; we derive it from the
finite ensembles and compute its Barnes normalizer. The transform
characterization is not used in the construction or geometric uniqueness.
The finite ensembles at integer parameters also have the truncated-unitary
interpretation of~\cite{ZS2000}, used in~\cite{Krishnapur2009};
Corollary~\ref{cor:unitary} recovers the corresponding function limit
from the present construction.

\subsection{Organization}\label{subsec:organization}

Section~\ref{sec:preliminaries} collects the finite-ensemble identities,
analytic coordinates, and scalar normalization.
Section~\ref{sec:construction} proves
Theorem~\ref{thm:realization} by bounded-statistic tightness, joint
logarithmic coordinates, and the vanishing-rotation argument. It then
establishes the radial product and all-center Green identities. Sections~\ref{sec:covariance} and~\ref{sec:geometric} prove
the covariance and classification parts of
Theorem~\ref{thm:geometric}. Section~\ref{sec:continuity-integer} proves
Theorem~\ref{thm:continuity}, and Section~\ref{sec:transforms} treats the
joint transforms. Appendix~\ref{sec:invariant-kernels} records the
invariant holomorphic-kernel classification with its proof and attribution.
Appendix~\ref{sec:direct-normalization} gives a direct Green-average
argument for the normalized uniqueness conclusion.

\section{Preliminaries}\label{sec:preliminaries}

We first collect the finite-ensemble identities, continuity facts, and
centered-statistic estimate needed for the function construction. We also
identify the scalar normalization here. Section~\ref{sec:construction} uses these inputs for the direct function
limit, and then proves the Green and analytic product identities.

\subsection{Finite ensembles and local point-process limits}
\label{subsec:finite-ensembles}

The normalized monomials in $L^2(\mu_\alpha)$ are
$e_{n,\alpha}(z)=\binom{n+\alpha}{n}^{1/2}z^n$, $n\ge0$.
They form an orthonormal basis of its analytic subspace: polar integration
gives
\[
 \int_\D |z|^{2n}d\mu_\alpha(z)=\alpha B(n+1,\alpha).
\]
Their complete and finite sums give~\eqref{eq:model}
and~\eqref{eq:finite-kernel}. In particular, the planar finite intensity
satisfies
\begin{equation}\label{eq:intensity-domination}
 0\le\rho_{N,\alpha}(z)\uparrow\rho_\alpha(z).
\end{equation}
The standard existence theorem for locally trace-class projection kernels
applies; see~\cite{HKPV2006}. All point processes used below are simple.
Because their intensities are absolutely continuous, they assign no points
to a fixed point or a fixed circle almost surely.

\begin{lemma}\label{lem:radial}
Let $\eta_n=\pi B(n+1,\alpha)$. The symmetric labeled density of
$\Xi_{N,\alpha}$ relative to $dm^N$ is
\begin{equation}\label{eq:finite-density}
 \frac{|\Delta(z_1,\ldots,z_N)|^2}{N!\prod_{n=0}^{N-1}\eta_n}
            \prod_{j=1}^N(1-|z_j|^2)^{\alpha-1}.
\end{equation}
Its unordered squared radii have the law of independent
$Y_j\sim\operatorname{Beta}(j,\alpha)$, $1\le j\le N$, after forgetting
the order. Consequently, with $X_{N,\alpha}=|F_{N,\alpha}(0)|^2$,
\begin{equation}\label{eq:radial-product}
 X_{N,\alpha}\overset d=N^\alpha\prod_{j=1}^NY_j.
\end{equation}
\end{lemma}
\begin{proof}
The rank-$N$ determinant density is the squared determinant of the
normalized monomials, divided by $N!$. This gives
\eqref{eq:finite-density}. Expanding its two Vandermonde determinants and
integrating the angles leaves only terms with equal permutations. The
substitution $t_j=|z_j|^2$ gives the symmetrization of the densities
$t^{j-1}(1-t)^{\alpha-1}/B(j,\alpha)$. This is also the finite-rank
specialization of~\cite[Theorem~26]{HKPV2006}. Evaluation of
\eqref{eq:finite-function} at zero gives~\eqref{eq:radial-product}.
\end{proof}

\begin{lemma}\label{lem:local-convergence}
One has $\Xi_{N,\alpha}\Longrightarrow H_\alpha$ in $\Conf$. If
$\alpha_j\to\alpha>0$, then $H_{\alpha_j}\Longrightarrow H_\alpha$.
For every compact $A\Subset\D$, the localized operators on
$L^2(\mu_\alpha)$ also satisfy
\begin{equation}\label{eq:local-trace}
 \norm{\ind_A(K_\alpha-K_{N,\alpha})\ind_A}_1\longrightarrow0.
\end{equation}
\end{lemma}
\begin{proof}
To compare parameters, use the common planar reference measure and set
\[
 \mathsf K_\alpha(z,w)=
 \sqrt{\tfrac{d\mu_\alpha}{dm}(z)}K_\alpha(z,w)
 \sqrt{\tfrac{d\mu_\alpha}{dm}(w)},
\]
with the analogous notation for finite rank. The monomial series converges
locally uniformly. When the parameter belongs to $[a,b]\Subset(0,\infty)$,
the coefficients have uniformly polynomial growth, which is dominated on
compact disks by the geometric decay of $(z\bar w)^n$. The kernels
therefore converge locally uniformly in both limits under consideration,
and the intensities on each compact set have a common integrable bound.

For completeness, let $v\in C_c(\D)$ be nonnegative and $h=1-e^{-v}$. The
Laplace functional is
\begin{equation}\label{eq:laplace-functional}
 \E e^{-\langle v,\xi\rangle}
 =\sum_{k=0}^\infty\frac{(-1)^k}{k!}
   \int_{\D^k}\prod_{j=1}^kh(z_j)
        \det[\mathsf K(z_i,z_j)]_{i,j=1}^k\,dm^k.
\end{equation}
Hadamard's inequality bounds the absolute value of the $k$th term by
$(\int_{\supp v}\rho\,dm)^k/k!$. Dominated convergence applies to this
series. Common local first-moment bounds give tightness of the counting
measures: along a compact exhaustion, choose successively increasing
bounds on local mass whose exceptional probabilities are summable.
Sets with these local mass bounds are relatively compact in the vague
topology. All cluster laws have the limiting Laplace functional, which
determines the law on $\Conf$.

Finally, $K_\alpha-K_{N,\alpha}$ is positive, and remains positive after
localization. Its trace norm equals its trace:
\[
 \norm{\ind_A(K_\alpha-K_{N,\alpha})\ind_A}_1
 =\sum_{n\ge N}\int_A|e_{n,\alpha}|^2d\mu_\alpha
 =\int_A(K_\alpha(z,z)-K_{N,\alpha}(z,z))d\mu_\alpha(z).
\]
The last expression tends to zero by dominated convergence on $A$.
\end{proof}

\subsection{Analytic coordinates and centered statistics}
\label{subsec:analytic-tools}

For $f(0)\ne0$, set
\begin{equation}\label{eq:jet-definition}
 L_k(f)=\frac{(\log f)^{(k)}(0)}{(k-1)!}
       =k[z^k]\log\frac{f(z)}{f(0)},\qquad k\ge1.
\end{equation}
Here the local logarithm on the right vanishes at zero; the derivatives do
not depend on any logarithmic branch choice. We call the finite vectors
$(L_1(f),\ldots,L_K(f))$ logarithmic jets.

The spaces $\Hol$ and $\Conf$ are Polish. A compatible complete metric on
$\Hol$ is obtained from the suprema on a countable compact-disk exhaustion.
The following facts specify how we pass from function limits to their
zeros and logarithmic coordinates.

\begin{lemma}\label{lem:analytic-coordinates}
Writing $a_n(f)=f^{(n)}(0)/n!$, one has
$\mathcal B(\Hol)=\sigma(a_0,a_1,\ldots)$. On the open set $f(0)\ne0$,
the vectors $(f(0),L_1(f),\ldots,L_K(f))$ and
$(a_0(f),\ldots,a_K(f))$ are related by continuous triangular maps.
The zero-divisor map $Z:\Hol\setminus\{0\}\to\Conf$ is continuous.
\end{lemma}
\begin{proof}
Each coefficient is continuous by Cauchy's formula. Conversely, point
evaluations are measurable limits of Taylor polynomials, and compact
suprema may be taken over a countable dense set. This proves the Borel
assertion. If $f/f(0)=\sum_{n\ge0}c_nz^n$ with $c_0=1$, then
\begin{equation}\label{eq:taylor-recursion}
 nc_n=\sum_{k=1}^nL_k(f)c_{n-k},\qquad n\ge1.
\end{equation}
It follows by differentiating the local logarithm. This recursion and the
power-series expansion of $\log(1+\sum_{n\ge1}c_nz^n)$ give the two
triangular maps; their denominators involve only powers of $f(0)$.

If $f_j\to f\not\equiv0$ locally uniformly, enclose the support of a
continuous compactly supported test function in a disk whose boundary has
no zeros of $f$. There are finitely many zeros inside that disk. Surround
them by disjoint small disks with zero-free boundaries. Rouch\'e's theorem
preserves the total multiplicity in each small disk, while a positive lower
bound for $|f|$ excludes other zeros on the remaining compact set.
Shrinking the small disks and using uniform continuity of the test
function proves vague convergence of the zero divisors.
\end{proof}

We will repeatedly retain auxiliary coordinates when removing cutoffs.
Here is the precise form of that approximation argument.

\begin{lemma}\label{lem:joint-approximation}
Suppose $(W_n,\xi_n)\Longrightarrow(W,\xi)$ in a product of Polish spaces.
Let $V_n$ and $V_{n,m}$ take values in another Polish space with compatible
metric $d$. Assume that, for each fixed $m$,
\[
 (W_n,\xi_n,V_{n,m})\Longrightarrow(W,\xi,V_m),
\]
that $V_m\to V$ in probability on the same probability space as
$(W,\xi)$, and that, for every $\varepsilon>0$,
\[
 \lim_{m\to\infty}\limsup_{n\to\infty}
       \Prob\{d(V_n,V_{n,m})>\varepsilon\}=0.
\]
Then $(W_n,\xi_n,V_n)\Longrightarrow(W,\xi,V)$.
\end{lemma}
\begin{proof}
Apply a bounded Lipschitz test function on the product. Replacing $V_n$ by
$V_{n,m}$ changes its expectation by at most its Lipschitz constant times
$\varepsilon$, plus twice its supremum norm times the displayed exceptional
probability. Pass first $n\to\infty$, then $m\to\infty$, and finally
$\varepsilon\downarrow0$. The other coordinates are unchanged throughout.
\end{proof}

For a projection process $\xi$ with planar kernel $\mathsf K$ and
intensity $\rho$, initially define
\[
 \X_\xi(f)=\sum_{x\in\xi}f(x)-\int_\D f(z)\rho(z)dm(z)
\]
for bounded compactly supported complex-valued $f$.

\begin{lemma}\label{lem:variance}
For a Hermitian locally trace-class projection kernel,
\begin{equation}\label{eq:variance-bound}
 \E|\X_\xi(f)|^2\le\int_\D|f|^2\rho\,dm.
\end{equation}
Consequently, $\X_\xi$ extends uniquely to a continuous linear map from
$L^2(\rho\,dm)$ into the centered random variables in $L^2$.
\end{lemma}
\begin{proof}
The first two correlation functions give
\[
 \E|\X_\xi(f)|^2=\int|f|^2\rho\,dm
 -\iint f(z)\overline{f(w)}|\mathsf K(z,w)|^2dm(z)dm(w).
\]
The subtracted term is the squared Hilbert--Schmidt norm of
$\mathsf K M_{\bar f}\mathsf K$, hence is nonnegative. Approximation by
bounded compactly supported functions proves the extension.
\end{proof}

This is a centered completion: the sum and its mean need not converge
separately. The general centered-statistic framework is discussed
in~\cite[Section~4.1]{BQ2017}. In this paper the elementary bound
\eqref{eq:variance-bound} suffices. For finitely many symbols it can be
summed coordinatewise, and it also applies to the finite projection
processes.

\subsection{Scalar normalization at the origin}\label{subsec:scalar}

The radial identity fixes both the scale $N^{\alpha/2}$ and the limiting
basepoint law. We record its full moment domain because negative moments
will be used in the change of measure and the logarithmic moment bounds.

\begin{proposition}\label{prop:scalar}
The absolute Mellin domain of $X_{N,\alpha}$ is
$\operatorname{Re}u>-1$, and on this domain
\begin{align}
 M_{N,\alpha}(u):=\E X_{N,\alpha}^u
 &=N^{\alpha u}\prod_{j=1}^N
   \frac{\Gamma(j+u)\Gamma(j+\alpha)}
        {\Gamma(j)\Gamma(j+\alpha+u)}\label{eq:finite-mellin}\\
 &=M_\alpha(u)N^{\alpha u}
   \frac{G(N+1+u)G(N+1+\alpha)}
        {G(N+1)G(N+1+\alpha+u)}.\label{eq:finite-barnes}
\end{align}
For every compact $J\Subset\{\operatorname{Re}u>-1\}$,
\begin{equation}\label{eq:balanced-barnes}
 M_{N,\alpha}(u)=M_\alpha(u)(1+O_{J,\alpha}(N^{-1})),
 \qquad u\in J,
\end{equation}
uniformly also when $\alpha$ belongs to a positive compact interval.
There is an almost surely positive finite random variable $X_\alpha$ such
that $X_{N,\alpha}\Longrightarrow X_\alpha$ and
\begin{equation}\label{eq:scalar-limit}
 \E X_\alpha^u=M_\alpha(u),\qquad \operatorname{Re}u>-1.
\end{equation}
For $\operatorname{Re}u\le-1$, $\E|X_\alpha^u|=\infty$. For every
$0<\delta<1$,
\begin{equation}\label{eq:log-moments}
 \sup_N\E e^{\delta|\log X_{N,\alpha}|}<\infty,
 \qquad \E e^{\delta|\log X_\alpha|}<\infty.
\end{equation}
\end{proposition}
\begin{proof}
The beta integral in~\eqref{eq:radial-product} gives
\eqref{eq:finite-mellin}. The most restrictive integrability condition is
that of $Y_1$, namely $\operatorname{Re}u>-1$.
The Barnes recursion gives~\eqref{eq:finite-barnes}.

To verify the normalization in its asymptotic form, let
$g(z)=\log G(z+1)$ in a sector containing the positive real axis. The
standard Barnes expansion~\cite[Section~5.17, Eq.~5.17.5]{DLMF} and
Cauchy's formula imply
\[
 g''(z)=\log z+O(|z|^{-2})
\]
uniformly in a smaller sector. Indeed, the leading terms of $g$ are
$\frac12z^2\log z-\frac34z^2+\frac12z\log(2\pi)
-\frac1{12}\log z$; the analytic remainder can be differentiated on the
smaller sector. Hence
\begin{align*}
 &g(N+u)+g(N+\alpha)-g(N)-g(N+\alpha+u)\\
 &\qquad=-\int_0^u\int_0^\alpha g''(N+s+t)\,dt\,ds
        =-\alpha u\log N+O_{J,\alpha}(N^{-1}).
\end{align*}
For complex $u$ the first integral follows the line segment from $0$ to
$u$. Both segments remain in the region of a common estimate for all
sufficiently large $N$. Exponentiating proves
\eqref{eq:balanced-barnes} and its parameter uniformity.

For positivity of the limit, take independent
$Y_j\sim\operatorname{Beta}(j,\alpha)$ on one probability space.
Differentiation of the beta integral, followed by the standard gamma
asymptotics~\cite[Section~5.11]{DLMF}, gives
\begin{align*}
 \E\log Y_j&=\psi(j)-\psi(j+\alpha)=-\alpha/j+O(j^{-2}),\\
 \Var(\log Y_j)&=\psi_1(j)-\psi_1(j+\alpha)=O(j^{-2}),
\end{align*}
where $\psi$ and $\psi_1$ are the digamma and trigamma functions. Thus the
centered series of logarithms converges almost surely and in $L^2$ by the
$L^2$ martingale convergence theorem. Its deterministic part also
converges, since the error after adding $\alpha/j$ is summable and
$\sum_{j\le N}j^{-1}-\log N$ converges. Consequently,
\[
 \alpha\log N+\sum_{j=1}^N\log Y_j\longrightarrow\Lambda_\alpha\in\R
 \quad\text{almost surely}.
\]
Set $X_\alpha=e^{\Lambda_\alpha}$. This coupling concerns only the radial
products; it does not assert a coupling of the original finite functions.

Given $u$ with $\operatorname{Re}u>-1$, choose $p>1$ with
$p\operatorname{Re}u>-1$. Formula~\eqref{eq:balanced-barnes} bounds the
$p$th absolute moment of the corresponding radial product raised to $u$.
Uniform integrability therefore proves~\eqref{eq:scalar-limit}, including
negative real parts. To prove maximality at the limit, separate the first
beta variable:
\[
 X_\alpha=Y_1Z_\alpha,
\]
where $0<Z_\alpha<\infty$ almost surely and $Z_\alpha$ is independent of
$Y_1$. The same logarithmic-series argument constructs $Z_\alpha$ from
$Y_j$, $j\ge2$. Since $\E Y_1^s=\infty$ for $s\le-1$, conditioning on
$Z_\alpha$ proves divergence for all such $s$. Finally,
$e^{\delta|\log x|}\le x^\delta+x^{-\delta}$ yields
\eqref{eq:log-moments}.
\end{proof}

\begin{remark}[Barnes beta and infinite beta-product identification]
\label{rem:barnes-beta}
The limiting scalar law is a previously studied infinite beta-product
law. In the notation of Letemplier and
Simon~\cite{LetemplierSimon2019},
\[
 X_\alpha\overset d=\Gamma(1+\alpha)\,T(1,1,\alpha).
\]
Equivalently, in the Barnes-beta notation of
Ostrovsky~\cite{Ostrovsky2016},
\[
 X_\alpha\overset d=
 \beta_{2,1}\bigl((1,1),(1,\alpha)\bigr).
\]
Indeed,
\[
 \prod_{j=1}^N\frac{j+\alpha}{j}
 =
 \frac{\Gamma(N+1+\alpha)}
      {\Gamma(1+\alpha)\Gamma(N+1)}
 \sim\frac{N^\alpha}{\Gamma(1+\alpha)},
\]
which gives the first identification from the radial beta product,
while the second follows by specializing the double-gamma Mellin
transform of $\beta_{2,1}$ and rewriting it in terms of the Barnes
$G$-function. 
\end{remark}

We shall use
\begin{equation}\label{eq:cN}
 c_{N,\alpha}=\frac12(\log M_{N,\alpha})'(0)\longrightarrow c_\alpha.
\end{equation}
The convergence follows from locally uniform analytic convergence in
\eqref{eq:balanced-barnes}. The function $c_\alpha$ is continuous for
$\alpha>0$. For later reference, negative moments also give, for $0<q<1$,
\begin{equation}\label{eq:basepoint-small}
 \sup_N\Prob\{|F_{N,\alpha}(0)|\le\varepsilon\}
       \le C_{\alpha,q}\varepsilon^{2q},\qquad 0<\varepsilon<1.
\end{equation}
In particular, evaluation at zero identifies the squared basepoint modulus
of any analytic subsequential limit with $X_\alpha$.

\section{Canonical realization and reconstruction}
\label{sec:construction}

We first prove tightness using a bounded statistic. The centered
logarithm at the origin and the logarithmic jets then identify the
normalized shape of every analytic cluster limit on its actual divisor.
A vanishing rotation determines the conditional phase. These arguments
prove Theorem~\ref{thm:realization} without a change of measure. We
subsequently obtain the product representation and all-center Green
identities needed for the geometric results.

\subsection{Tightness from a bounded statistic}
\label{subsec:two-inputs}\label{subsec:bounded-tightness}

The bounded statistic used below already appears in Krishnapur's
tightness argument~\cite[proof of Lemma~14, Section~7]{Krishnapur2009}.
For the present finite ensembles, the determinantal variance contraction
and the subharmonic mean inequality give the following compact-open
estimate for every $\alpha>0$.

\begin{proposition}\label{prop:finite-tightness}
The laws of $F_{N,\alpha}$ are tight in $\Hol$. More precisely, for every
$r<1$ and $M>1$,
\begin{equation}\label{eq:tightness-rate}
 \sup_N\Prob\left\{\sup_{|z|\le r}|F_{N,\alpha}(z)|>M\right\}
             \le\frac{C_{\alpha,r}}{\log M}.
\end{equation}
\end{proposition}
\begin{proof}

For $z,w\in\D$, put
\begin{equation}\label{eq:bounded-defect}
 D_z(w)=\frac{(1-|z|^2)(1-|w|^2)}{|1-z\bar w|^2}
       =1-\left|\frac{z-w}{1-z\bar w}\right|^2,
 \qquad Q_N(z)=\sum_{x\in\Xi_{N,\alpha}}D_z(x).
\end{equation}
Since $\log(1-v)\le-v$ for $0\le v<1$, the finite product gives
\begin{equation}\label{eq:defect-log-bound}
 \log|F_{N,\alpha}(z)|^2\le\alpha\log N-Q_N(z).
\end{equation}
At a zero the left side is $-\infty$. The geometric invariance of the
pseudohyperbolic distance and $dm(w)/(1-|w|^2)^2$ gives
\[
 \int_\D D_z(w)^2d\nu_\alpha(w)
 =\int_\D(1-|w|^2)^2d\nu_\alpha(w)=\alpha.
\]
Consequently, intensity domination and Lemma~\ref{lem:variance} yield
$\Var Q_N(z)\le\alpha$, uniformly in both $N$ and $z$.

It remains to estimate the mean. Write $t=|z|^2$. Angular averaging gives
\[
 \frac1{2\pi}\int_0^{2\pi}D_z(\sqrt y e^{i\theta})\,d\theta
 =\frac{(1-t)(1-y)}{1-ty}
 =(1-y)-\frac{t(1-y)^2}{1-ty}.
\]
The radial intensity and Lemma~\ref{lem:radial}, with
$Y_j\sim\operatorname{Beta}(j,\alpha)$, give
\begin{equation}\label{eq:defect-mean}
 \E Q_N(z)=\sum_{j=1}^N\frac{\alpha}{j+\alpha}
       -t\sum_{j=1}^N\E\frac{(1-Y_j)^2}{1-tY_j}.
\end{equation}
The second beta moment telescopes:
\[
 \sum_{j=1}^N\E(1-Y_j)^2
 =\alpha(\alpha+1)\sum_{j=1}^N
       \left(\frac1{j+\alpha}-\frac1{j+\alpha+1}\right)
 \le\alpha.
\]
Since $0<\sum_{j=1}^Nj^{-1}-\log N\le1$, for every $s<1$ we obtain
\[
 \sup_N\sup_{|z|\le s}|\E Q_N(z)-\alpha\log N|
 \le b_{\alpha,s}:=\alpha+\frac{\alpha^2\pi^2}{6}
                              +\frac{\alpha s^2}{1-s^2}.
\]
Together with~\eqref{eq:defect-log-bound} and the variance bound, this gives
\begin{equation}\label{eq:log-plus-bound}
 \sup_N\sup_{|z|\le s}\E\log^+|F_{N,\alpha}(z)|
                  \le\tfrac12(b_{\alpha,s}+\sqrt\alpha)<\infty.
\end{equation}

For $r<s$, the mean inequality for the nonnegative subharmonic function
$u_N=\log^+|F_{N,\alpha}|$ gives
\[
 \sup_{|z|\le r}u_N(z)
       \le C_{r,s}\int_{D_s}u_N(w)dm(w).
\]
Taking expectations and applying Markov's inequality proves
\eqref{eq:tightness-rate}. Along a countable disk exhaustion, choose
successive bounds whose exceptional probabilities sum to less than a
given $\varepsilon$. The analytic functions satisfying these bounds form
a compact set by Montel's theorem. This proves tightness in $\Hol$.
\end{proof}

\subsection{Joint basepoint and logarithmic coordinates}
\label{subsec:unweighted-coordinates}

Set $q_0(w)=-\log|w|$ for $w\ne0$ and $q_0(0)=0$. Direct integration gives
\begin{equation}\label{eq:basepoint-green-norm}
 \int_\D q_0^2d\nu_\alpha
 =\frac\alpha4\int_0^1\frac{(\log t)^2}{(1-t)^2}\,dt
 =\frac{\alpha\pi^2}{12}.
\end{equation}
Let $T_{0,N}=\X_{\Xi_{N,\alpha}}(q_0)$. Taking the logarithm at the
origin and then its expectation gives
\begin{equation}\label{eq:cN-green}
 \log|F_{N,\alpha}(0)|=c_{N,\alpha}-T_{0,N},\qquad
 c_{N,\alpha}=\frac\alpha2\log N-\int_\D q_0\rho_{N,\alpha}\,dm.
\end{equation}
This is the constant in~\eqref{eq:cN}. For $k\ge1$, set
$h_k(w)=\bar w^k-w^{-k}$ for $w\ne0$, with $h_k(0)=0$.
Expanding the normalized logarithm of each finite factor gives
\begin{equation}\label{eq:finite-jet-sum}
 L_k(F_{N,\alpha})=\sum_{x\in\Xi_{N,\alpha}}h_k(x).
\end{equation}

\begin{lemma}\label{lem:unweighted-coordinates}
Fix $R_n=1-2^{-4n}$. There are total Borel maps
$T_0:\Conf\to\R$ and $J_k:\Conf\to\C$, $k\ge1$, such that, for
$\Xi\sim H_\alpha$, almost surely
\begin{align}
 T_0(\Xi)&=\lim_{n\to\infty}
       \left\{\sum_{\substack{x\in\Xi\\|x|<R_n}}q_0(x)
                       -\int_{D_{R_n}}q_0\,d\nu_\alpha\right\},
                       \label{eq:canonical-basepoint}\\
 J_k(\Xi)&=\lim_{n\to\infty}
                 \sum_{\substack{x\in\Xi\\|x|<R_n}}h_k(x),
                 \qquad k\ge1.\label{eq:unweighted-jet-limit}
\end{align}
The first variable is the $L^2$ completion $\X_\Xi(q_0)$. For every
fixed $K$, jointly with the configuration,
\begin{equation}\label{eq:unweighted-joint-limit}
 (\Xi_{N,\alpha},T_{0,N},L_1(F_{N,\alpha}),\ldots,L_K(F_{N,\alpha}))
 \Longrightarrow(\Xi,T_0(\Xi),J_1(\Xi),\ldots,J_K(\Xi)).
\end{equation}
If $(F_{N,\alpha},\Xi_{N,\alpha})\Longrightarrow(F,\Xi)$ along a
subsequence, the same convergence holds with the function coordinate
retained on both sides.
\end{lemma}
\begin{proof}
For $R\ge1/2$, $q_0(w)\le C(1-|w|^2)$ on $|w|>R$. Hence
\[
 \int_{|w|>R}q_0(w)^2d\nu_\alpha(w)\le C_\alpha(1-R^2).
\]
The variance contraction proves $L^2$ convergence of the centered radial
sums to $\X_\Xi(q_0)$. The squared errors along $(R_n)$ are summable,
so Chebyshev's inequality and Borel--Cantelli give the almost sure limit
in~\eqref{eq:canonical-basepoint}.

Fix $r_0\in(0,1)$. On $|w|>r_0$, one has
$|h_k(w)|\le C_{k,r_0}(1-|w|^2)$. Every radial annulus has zero mean
for this symbol. Lemma~\ref{lem:variance} therefore gives, for
$r_0<R<S<1$,
\begin{equation}\label{eq:unweighted-jet-tail}
 \E\left|\sum_{\substack{x\in\Xi\\R<|x|<S}}h_k(x)\right|^2
 \le C_{\alpha,k,r_0}(S^2-R^2).
\end{equation}
The same estimate holds uniformly for the finite processes. The exterior
radial sums are Cauchy in $L^2$. Their squared errors at $R_n$ are
summable, giving an almost sure limit along that sequence. The interior
contribution is a literal finite sum, since the configuration is locally
finite and has no point at zero. It need not have a second moment.
Intersecting over $k$ gives~\eqref{eq:unweighted-jet-limit} on one event
of probability one. The convergence set of these scalar Borel sums is
Borel. Use the limits on this set, excluding configurations containing
zero or a point on an $R_n$-circle, and assign zero to all maps on its
complement.

To prove joint convergence, approximate $q_0$ by continuous compactly
supported symbols in $L^2(\nu_\alpha)$. The variance contraction and
intensity domination control the approximation errors uniformly in $N$.
For a jet, first restrict $h_k$ to $\varepsilon<|w|<R$. The truncated
symbol is bounded and compactly supported; its sum over a configuration
is continuous at configurations with no point on the two boundary
circles. The exterior error tends to zero in probability uniformly in
$N$ by~\eqref{eq:unweighted-jet-tail}. At the origin use instead
\begin{equation}\label{eq:jet-origin-cutoff}
 \sup_N\Prob\{\Xi_{N,\alpha}(D_\varepsilon)>0\}
 \le\nu_\alpha(D_\varepsilon)\le C_\alpha\varepsilon^2,
 \qquad 0<\varepsilon\le1/2,
\end{equation}
and the identical bound for $\Xi$. On the complement of this event
the interior cutoff changes none of the jets. Thus the same cutoffs
control any fixed finite block without an integrability assumption on
the interior singular sums. Vague convergence, local convergence of the
intensities, and Lemma~\ref{lem:joint-approximation} prove
\eqref{eq:unweighted-joint-limit}. That lemma also retains any convergent
auxiliary function coordinate throughout the cutoff argument.
\end{proof}

\subsection{Proof of Theorem~\ref{thm:realization}}
\label{subsec:main-construction-proof}

\begin{proof}
By Proposition~\ref{prop:finite-tightness} and
Lemma~\ref{lem:local-convergence}, the joint function--configuration laws
are tight. Let $(F,\Xi)$ be a joint subsequential limit. Then
$\Xi\sim H_\alpha$, and Proposition~\ref{prop:scalar} and continuous
evaluation give $|F(0)|^2\overset d=X_\alpha>0$ almost surely. The
continuity of the zero-divisor map now gives $Z(F)=\Xi$.
Lemma~\ref{lem:unweighted-coordinates}, with the function coordinate
retained, and the continuous logarithmic coordinate maps imply
\begin{equation}\label{eq:cluster-coordinates}
 \log|F(0)|=c_\alpha-T_0(\Xi),\qquad
 L_k(F)=J_k(\Xi),\quad k\ge1,
\end{equation}
on one event of probability one.

These coordinates determine a Borel normalized analytic shape. To make
this assertion explicit, apply recursion~\eqref{eq:taylor-recursion}
with $c_0=1$ and $L_k=J_k(\xi)$ to define scalar Borel coefficients
$c_n(\xi)$. The condition
\[
 \sum_{n\ge0}|c_n(\xi)|r^n<\infty
 \quad\text{for every rational }r\in(0,1)
\]
defines a Borel set, on which the Taylor series is a Borel
$\Hol$-valued map $P_\xi$. For any cluster limit, this series is the
Taylor series of $F/F(0)$, so the set has full $H_\alpha$ measure and
$Z(P_\Xi)=\Xi$ almost surely. Intersect it with the common convergence
set in Lemma~\ref{lem:unweighted-coordinates} and with the Borel
condition $Z(P_\xi)=\xi$. Denote the resulting full set by
$\mathscr G_\alpha$. On its complement set $P_\xi\equiv1$ and
$T_0(\xi)=J_k(\xi)=0$. These conventions change no limiting law.
This construction depends only on the coordinate maps and is therefore
the same for every subsequence. Subsequent restrictions of
$\mathscr G_\alpha$ by Borel null sets will preserve these versions
almost surely and all the asserted identities of laws.

Writing $A_\alpha(\xi)=e^{c_\alpha-T_0(\xi)}P_\xi$, we have
\begin{equation}\label{eq:cluster-shape}
 F=UA_\alpha(\Xi),\qquad U=F(0)/|F(0)|.
\end{equation}
It remains to identify the conditional phase. The finite ensemble is
rotation invariant, and its product satisfies the exact identity
\begin{equation}\label{eq:finite-rotation}
 (F_{N,\alpha},\Xi_{N,\alpha})\overset d=
 \bigl(e^{iNt}F_{N,\alpha}(e^{-it}\,\cdot),e^{it}\Xi_{N,\alpha}\bigr),
 \qquad t\in\R.
\end{equation}
Fix $\theta\in\R$ and take $t=\theta/N$. The map
\begin{equation}\label{eq:rotation-map}
 (s,f,\xi)\longmapsto
       \bigl(e^{i\theta}f(e^{-i\theta s}\,\cdot),e^{i\theta s}\xi\bigr)
\end{equation}
is jointly continuous at $s=0$. For the measure coordinate, test against
compactly supported continuous functions: small rotations keep the
supports in one compact set, where vague convergence bounds local mass.
The continuous mapping theorem applied to
$(1/N,F_{N,\alpha},\Xi_{N,\alpha})$ yields
\begin{equation}\label{eq:joint-phase-invariance}
 (F,\Xi)\overset d=(e^{i\theta}F,\Xi)
 \qquad\text{for each }\theta.
\end{equation}
Thus $(U,\Xi)\overset d=(e^{i\theta}U,\Xi)$. For bounded Borel
functions $a$ on $\T$ and $b$ on $\Conf$, integration in $\theta$
and Fubini's theorem give
\begin{equation}\label{eq:haar-average}
 \E[a(U)b(\Xi)]
       =\E b(\Xi)\int_\T a(v)\,d\Haar(v).
\end{equation}
Hence $U$ is Haar and independent of the entire configuration. No common
samplewise event indexed by all rotations is required.

Every joint cluster law is consequently the law of
$(UA_\alpha(\Xi),\Xi)$ for independent $\Xi\sim H_\alpha$ and
$U\sim\Haar(\T)$. Tightness proves full-sequence convergence and
defines $F_\alpha$ with this law. Equations~\eqref{eq:cluster-coordinates}
and~\eqref{eq:cluster-shape} give the actual-divisor reconstruction and
the converse sampling statement. Finally, Proposition~\ref{prop:scalar}
identifies the basepoint Mellin transform, including its maximal domain.
\end{proof}

\begin{corollary}\label{cor:finite-jets}
For every fixed $K\ge1$,
\begin{equation}\label{eq:main-jets}
 (F_{N,\alpha}(0),L_1(F_{N,\alpha}),\ldots,L_K(F_{N,\alpha}))
 \Longrightarrow(F_\alpha(0),L_1(F_\alpha),\ldots,L_K(F_\alpha)).
\end{equation}
The phase $F_\alpha(0)/|F_\alpha(0)|$ is Haar and independent of
$(|F_\alpha(0)|,L_1(F_\alpha),\ldots,L_K(F_\alpha))$.
\end{corollary}
\begin{proof}
Apply Lemma~\ref{lem:analytic-coordinates} at the nonzero limiting
basepoint. The modulus and jets are functions of the divisor by
\eqref{eq:cluster-coordinates}, so~\eqref{eq:haar-average} proves the
independence assertion.
\end{proof}

\subsection{Green estimates and the finite modulus identity}
\label{subsec:green-estimates}

We now extend the basepoint statistic to other centers. These identities
will give the explicit radial product and the geometric transport law.

For $z\in\D$, define the Green symbol
\begin{equation}\label{eq:green-symbol}
 q_z(w)=-\log\left|\frac{z-w}{1-\bar zw}\right|,\qquad w\ne z.
\end{equation}
For the purpose of integration, set $q_z(z)=0$. Also write
\begin{equation}\label{eq:green-drift}
 I_\alpha(z)=-\frac\alpha2\log(1-|z|^2),
 \qquad d_\alpha(z)=c_\alpha+I_\alpha(z).
\end{equation}
The estimates below show that $q_z\in L^2(\nu_\alpha)$. Thus
Lemma~\ref{lem:variance} defines
$T_z^{(2)}(\Xi)=\X_\Xi(q_z)$ for $\Xi\sim H_\alpha$, as an $L^2$
random variable. The superscript distinguishes this completion from the
jointly measurable, all-center version constructed below.

For the finite process let
\begin{equation}\label{eq:finite-centered-green}
 T_{z,N}=\X_{\Xi_{N,\alpha}}(q_z)
       =\sum_{x\in\Xi_{N,\alpha}}q_z(x)
             -\int_\D q_z(w)\rho_{N,\alpha}(w)dm(w).
\end{equation}
The finite intensity has total mass $N$, so the integrability supplied by
the $L^2$ bound also justifies this literal centered finite sum.

\begin{lemma}\label{lem:green-estimates}
For every $z\in\D$,
\begin{equation}\label{eq:green-norm}
 \int_\D q_z(w)^2d\nu_\alpha(w)=\frac{\alpha\pi^2}{12}.
\end{equation}
For each compact $B\Subset\D$ and all $R$ sufficiently close to $1$,
\begin{equation}\label{eq:green-tail}
 \sup_{z\in B}\int_{|w|>R}q_z(w)^2d\nu_\alpha(w)
                   \le C_{\alpha,B}(1-R^2).
\end{equation}
The map $z\mapsto q_z$ is continuous from $\D$ into $L^2(\nu_\alpha)$.
\end{lemma}
\begin{proof}
The pseudohyperbolic distance and $dm(w)/(1-|w|^2)^2$ are invariant under
disk automorphisms. Moving $z$ to zero therefore gives
\[
 \int q_z^2d\nu_\alpha=\int q_0^2d\nu_\alpha
                         =\frac{\alpha\pi^2}{12}
\]
by~\eqref{eq:basepoint-green-norm}. The geometric
invariance used here also follows directly from
\begin{equation}\label{eq:pseudohyperbolic}
 1-\left|\frac{z-w}{1-\bar zw}\right|^2
       =\frac{(1-|z|^2)(1-|w|^2)}{|1-\bar zw|^2}.
\end{equation}
For $z\in B$ and $|w|$ near $1$, this identity implies
$q_z(w)\le C_B(1-|w|^2)$. Its square cancels the boundary factor in
$\rho_\alpha$, proving~\eqref{eq:green-tail}.

On a fixed compact subdisk, the intensity is bounded. The contribution
from an $\varepsilon$-disk about the center is bounded uniformly for
centers in that subdisk by
\[
 C\int_0^\varepsilon(1+|\log r|^2)r\,dr
       \le C'\varepsilon^2(1+|\log\varepsilon|^2).
\]
To prove continuity in the center, first remove a boundary annulus using
\eqref{eq:green-tail}, then remove small neighborhoods of the two nearby
centers using this estimate. On the remaining compact set the symbols
converge uniformly. Letting the two removed contributions tend to zero
proves the assertion.
\end{proof}

The deterministic part of the finite product can be computed by angular
averaging rather than by a separate expansion of its kernel.

\begin{lemma}\label{lem:finite-modulus}
For $0<r<1$,
\begin{equation}\label{eq:angular-mean}
 \frac1{2\pi}\int_0^{2\pi}(q_0-q_z)(re^{it})\,dt
       =\log\frac{\max\{|z|,r\}}{r}.
\end{equation}
Thus, for $R>|z|$,
\begin{equation}\label{eq:mean-drift}
 \int_{D_R}(q_0-q_z)\,d\nu_\alpha=I_\alpha(z).
\end{equation}
Let
\begin{equation}\label{eq:finite-drift}
 I_{N,\alpha}(z)=\int_\D(q_0-q_z)(w)\rho_{N,\alpha}(w)dm(w).
\end{equation}
Then $0\le I_{N,\alpha}\le I_\alpha$ and
$I_{N,\alpha}\to I_\alpha$ locally uniformly. For every $z$ outside the
finite configuration,
\begin{equation}\label{eq:finite-modulus}
 \log|F_{N,\alpha}(z)|
       =c_{N,\alpha}+I_{N,\alpha}(z)-T_{z,N},
\end{equation}
with $c_{N,\alpha}$ as in~\eqref{eq:cN-green}.
\end{lemma}
\begin{proof}
The angular means of $\log|z-re^{it}|$ and
$\log|1-\bar zre^{it}|$ are $\log\max\{|z|,r\}$ and zero,
respectively. This proves~\eqref{eq:angular-mean}. Setting $s=|z|$, the
integral in~\eqref{eq:mean-drift} is
\[
 2\alpha\int_0^s\frac{r\log(s/r)}{(1-r^2)^2}\,dr.
\]
It is zero at $s=0$ and its derivative is $\alpha s/(1-s^2)$, so it
is $I_\alpha(z)$. This is a radial-cutoff identity; absolute
integrability of $q_0-q_z$ against the infinite intensity is not needed.

All finite intensities are radial. Angular averaging and
\eqref{eq:intensity-domination} give the inequalities for $I_{N,\alpha}$.
For $|z|\le s<1$, the difference from $I_\alpha$ is bounded by
\[
 2\pi\int_0^s\log(s/r)
            (\rho_\alpha(r)-\rho_{N,\alpha}(r))r\,dr,
\]
which tends to zero by dominated convergence. The logarithm of the finite
product is $\frac\alpha2\log N-\sum_x q_z(x)$.
Adding and subtracting its mean, and using~\eqref{eq:cN-green},
proves~\eqref{eq:finite-modulus}.
\end{proof}

\subsection{Radial product and Green reconstruction}
\label{subsec:canonical-proof}

\begin{proposition}\label{prop:canonical}
After restricting $\mathscr G_\alpha$ by a Borel $H_\alpha$-null set,
the maps of Theorem~\ref{thm:realization} satisfy, for $\xi\in\mathscr G_\alpha$,
\begin{equation}\label{eq:canonical-product}
 P_\xi(z)=\lim_{n\to\infty}
       \prod_{\substack{x\in\xi\\|x|<R_n}}
                       \frac{1-z/x}{1-z\bar x}
                       \quad\text{in }\Hol,
 \qquad R_n=1-2^{-4n},
\end{equation}
and $T_0$ is given by~\eqref{eq:canonical-basepoint}.
On this set, $0\notin\xi$, $P_\xi(0)=1$, and $Z(P_\xi)=\xi$.
On its complement take $P_\xi\equiv1$ and $T_0(\xi)=J_k(\xi)=0$.
The total function
\begin{equation}\label{eq:canonical-green}
 T_z(\xi)=
 \begin{cases}
 T_0(\xi)-\log|P_\xi(z)|+I_\alpha(z),&z\notin\xi,\\
 0,&z\in\xi
 \end{cases}
\end{equation}
is jointly Borel and is continuous in $z$ off $\xi$. For every fixed
$z$, $T_z(\Xi)=T_z^{(2)}(\Xi)$ almost surely. In particular, if
\begin{equation}\label{eq:canonical-A}
 A_\alpha(\xi)=e^{c_\alpha-T_0(\xi)}P_\xi,
\end{equation}
then, on $\mathscr G_\alpha$, simultaneously for every $z\notin\xi$,
\begin{equation}\label{eq:canonical-modulus}
 \log|A_\alpha(\xi)(z)|=d_\alpha(z)-T_z(\xi).
\end{equation}
For all $k\ge1$, the logarithmic jets are the radial limits
\begin{equation}\label{eq:canonical-jets}
 L_k(P_\xi)=\lim_{n\to\infty}
     \sum_{\substack{x\in\xi\\|x|<R_n}}(\bar x^k-x^{-k}).
\end{equation}
The logarithmic derivative is given, locally uniformly off $\xi$, by
\begin{equation}\label{eq:canonical-connection}
 \frac{P_\xi'(z)}{P_\xi(z)}
 =\lim_{n\to\infty}\sum_{\substack{x\in\xi\\|x|<R_n}}
       \left(\frac1{z-x}+\frac{\bar x}{1-z\bar x}\right).
\end{equation}
\end{proposition}

The value $0$ in the second branch of~\eqref{eq:canonical-green} is only
a convention at an actual zero. The modulus identity is asserted off the
configuration; at a zero the analytic function vanishes. On an exceptional
configuration the first branch is well defined as $I_\alpha(z)$, since
$P_\xi\equiv1$ and $T_0(\xi)=0$.

The normalized finite factor is analytic in $z$ and equals one at the
origin. Far from the center, its analytic logarithm has zero angular mean.
The following estimate upgrades the fixed-center variance control to
local uniform convergence of analytic functions and all derivatives.

\begin{lemma}\label{lem:analytic-tail}
Fix $0<r<s<R_0<1$. For $|w|>R_0$ and $|z|<s$, let
\begin{equation}\label{eq:analytic-symbol}
 \ell_z(w)=\Log(1-z/w)-\Log(1-z\bar w)
      =\sum_{k=1}^\infty\frac{z^k}{k}(\bar w^k-w^{-k}),
\end{equation}
where both logarithms vanish at $z=0$. For $R_0\le R<S<1$ and
$\Xi\sim H_\alpha$, write
$H_{R,S}(z)=\sum_{R<|x|<S}\ell_z(x)$. For every integer $j\ge0$,
\begin{equation}\label{eq:analytic-tail-bound}
 \E\norm{H_{R,S}}_{C^j(\overline D_r)}^2
       \le C_{\alpha,r,s,R_0,j}(S^2-R^2)
       \le C_{\alpha,r,s,R_0,j}(1-R^2).
\end{equation}
The same estimate holds for $\Xi_{N,\alpha}$ with a constant independent
of $N$. Here the $C^j$ norm controls the first $j$ complex derivatives as
well as the function itself.
\end{lemma}
\begin{proof}
Since $1-|w|^{2k}\le k(1-|w|^2)$, the series gives
\begin{equation}\label{eq:analytic-symbol-bound}
 \sup_{|z|\le s}|\ell_z(w)|
       \le(1-|w|^2)\sum_{k=1}^\infty(s/R_0)^k
       \le C_{s,R_0}(1-|w|^2).
\end{equation}
Each term in~\eqref{eq:analytic-symbol} has zero angular mean in $w$.
Uniform convergence on the indicated region and the radial intensity
therefore imply $\E H_{R,S}(z)=0$. By Lemma~\ref{lem:variance} and
\eqref{eq:analytic-symbol-bound},
\[
 \sup_{|z|\le s}\E|H_{R,S}(z)|^2\le C(S^2-R^2).
\]
For an analytic function $h$ on $D_s$, the subharmonic mean inequality on
an intermediate disk, followed by Cauchy's formula for derivatives,
yields
\[
 \norm h_{C^j(\overline D_r)}^2
       \le C_{r,s,j}\int_{D_s}|h(z)|^2dm(z).
\]
Taking expectations proves~\eqref{eq:analytic-tail-bound}. The finite
processes have radial intensities dominated by $\rho_\alpha$, so the
same argument applies to them.
\end{proof}

\begin{proof}[Proof of Proposition~\ref{prop:canonical}]
Denote the finite product in~\eqref{eq:canonical-product} by
$P_{\xi,n}$, with value $1$ when $0\in\xi$. For a fixed disk
$\overline D_r$, choose $r<s<R_0<1$. Lemma~\ref{lem:analytic-tail}
implies, for every $j\ge0$,
\[
 \sum_n\E\norm{H_{R_n,R_{n+1}}}_{C^j(\overline D_r)}
       \le C\sum_n(1-R_n^2)^{1/2}<\infty,
\]
where finitely many terms with $R_n<R_0$ are omitted. Tonelli's theorem
therefore gives almost sure absolute convergence of the exterior
logarithmic increments, with all derivatives. Intersect over a countable
disk exhaustion and all $j$. On each disk the products split as
\[
 P_{\Xi,n}(z)=
 \prod_{|x|\le R_0}\frac{1-z/x}{1-z\bar x}
       \exp\left\{\sum_{R_0<|x|<R_n}\ell_z(x)\right\}.
\]
Their limits agree on overlapping disks and define an analytic function
$\widehat P_\Xi$ with value $1$ at zero and divisor $\Xi$. The exterior
factor is zero-free. In a random zero-free neighborhood of zero,
local uniform convergence permits passage to the logarithmic jets.
By~\eqref{eq:analytic-symbol} and~\eqref{eq:unweighted-jet-limit},
$L_k(\widehat P_\Xi)=J_k(\Xi)$ for every $k$. The Taylor recursion
therefore identifies $\widehat P_\Xi=P_\Xi$. This proves
\eqref{eq:canonical-product} and~\eqref{eq:canonical-jets}. On a compact
set avoiding $\Xi$, the limit has modulus bounded away from zero, so
convergence of the products and their derivatives gives
\eqref{eq:canonical-connection}.

For completeness, these identities may be imposed on a common Borel
full set. Finite products over a compact disk are Borel maps into
$\Hol$, by measurable enumeration and the point-evaluation description
of its Borel structure. Since $\Hol$ is completely metrizable, their
convergence set and limit are Borel. Intersect the previous
$\mathscr G_\alpha$ with this set and the countably many exterior
convergence events just used. Reset $P_\xi\equiv1$ and
$T_0(\xi)=J_k(\xi)=0$ on the complement. These changes are on an
$H_\alpha$-null set and preserve Theorem~\ref{thm:realization}.
Write $L_k(\xi)=L_k(P_\xi)$ and $S_\xi=P_\xi'/P_\xi$ on the good
set, with zero values outside it. Their evaluations are Borel wherever
defined. This fixes the versions used below.

For a fixed deterministic center $z$, the variance contraction and
\eqref{eq:green-tail} give
\begin{equation}\label{eq:radial-green-error}
 \E\left|\X_\Xi(q_z\ind_{D_{R_n}})-T_z^{(2)}(\Xi)\right|^2
                     \le C_{\alpha,z}(1-R_n^2)
\end{equation}
for all large $n$. Summability gives almost sure convergence along
$(R_n)$. For $R_n>|z|$, angular averaging in~\eqref{eq:mean-drift}
gives the exact identity
\begin{equation}\label{eq:radial-P-modulus}
 \log|P_{\Xi,n}(z)|
       =\X_\Xi(q_0\ind_{D_{R_n}})
        -\X_\Xi(q_z\ind_{D_{R_n}})+I_\alpha(z).
\end{equation}
Almost surely $z\notin\Xi$, so passing to the limit yields
\begin{equation}\label{eq:P-modulus-fixed}
 \log|P_\Xi(z)|=T_0(\Xi)-T_z^{(2)}(\Xi)+I_\alpha(z).
\end{equation}
Thus~\eqref{eq:canonical-green} is a version of $T_z^{(2)}$ for each
fixed $z$.

The incidence set $\{(\xi,z):z\in\xi\}$ is Borel, as follows by
testing the point measure on balls with radii decreasing to zero around
the variable center. On the good set, $P_\xi$ vanishes precisely at
$\xi$; outside it $P_\xi\equiv1$. Joint evaluation of the Borel
$\Hol$-valued map therefore makes~\eqref{eq:canonical-green} everywhere
well defined and jointly Borel, with continuous paths off $\xi$.
Its definition gives~\eqref{eq:canonical-modulus} simultaneously off
the configuration.
\end{proof}

\subsection{Joint Green limits}\label{subsec:cluster-proof}

The following extension retains the coupling established by the construction.

\begin{proposition}\label{prop:cluster}
If a subsequence satisfies
$(F_{N,\alpha},\Xi_{N,\alpha})\Longrightarrow(F,\Xi)$, then, for any
finite collection of deterministic centers $z_1,\ldots,z_\ell$,
\begin{equation}\label{eq:joint-green-limit}
 (F_{N,\alpha},\Xi_{N,\alpha},T_{z_1,N},\ldots,T_{z_\ell,N})
 \Longrightarrow(F,\Xi,T_{z_1}(\Xi),\ldots,T_{z_\ell}(\Xi)).
\end{equation}
Moreover, for each deterministic $z\in\D$,
\begin{equation}\label{eq:cluster-modulus}
 |F(z)|=\exp\{d_\alpha(z)-T_z(\Xi)\}\quad\text{almost surely}.
\end{equation}
In particular, $F(0)\ne0$ and $Z(F)=\Xi$ almost surely.
\end{proposition}

\begin{proof}
Let $(F_{N,\alpha},\Xi_{N,\alpha})\Longrightarrow(F,\Xi)$ along a
subsequence. For each of the finitely many centers $z_i$, choose
$f_{i,n}\in C_c(\D)$ approximating $q_{z_i}$ in $L^2(\nu_\alpha)$.
Such approximations can be obtained by removing a small disk about the
center and an outer boundary annulus, then taking a continuous cutoff.
Lemma~\ref{lem:green-estimates} gives control of both removed parts.
For fixed $n$, vague convergence and local convergence of the intensities
give, jointly with the retained function and configuration,
\[
 \X_{\Xi_{N,\alpha}}(f_{i,n})\Longrightarrow\X_\Xi(f_{i,n}),
 \qquad 1\le i\le \ell.
\]
By the variance contraction,
\begin{align*}
 \sup_N\E\bigl|T_{z_i,N}-\X_{\Xi_{N,\alpha}}(f_{i,n})\bigr|^2
       &\le\norm{q_{z_i}-f_{i,n}}_{L^2(\nu_\alpha)}^2,\\
 \E\bigl|T_{z_i}(\Xi)-\X_\Xi(f_{i,n})\bigr|^2
       &\le\norm{q_{z_i}-f_{i,n}}_{L^2(\nu_\alpha)}^2.
\end{align*}
These errors tend to zero; summing over the finite collection of centers
and applying Lemma~\ref{lem:joint-approximation} proves
\eqref{eq:joint-green-limit}.

For a fixed center, write the finite modulus identity in exponential
form,
\[
 |F_{N,\alpha}(z)|=
      \exp\{c_{N,\alpha}+I_{N,\alpha}(z)-T_{z,N}\}.
\]
In the joint convergence just established, the constants converge to
$d_\alpha(z)$. The relation between function evaluation and the real
Green coordinate is closed, so the limit satisfies
\eqref{eq:cluster-modulus}. The nonzero basepoint and the identity $Z(F)=\Xi$ also follow from
Theorem~\ref{thm:realization}.
\end{proof}

In particular, any finite collection of Green statistics may be included
jointly with the basepoint and logarithmic jets in the function limit.
Keep the $\Hol$ coordinate in~\eqref{eq:joint-green-limit} and apply the
finite continuous coordinate map at the nonzero limiting basepoint.
All resulting finite-dimensional laws are projections of the same joint
law, rather than separately specified limits.

\section{M\"obius covariance}\label{sec:covariance}

The covariance proof uses the invariance of the centered Green statistic.
Once its change of variables has been established for the fixed canonical
version, equality of moduli determines the analytic transformation up to a
constant phase. The conditional Haar phase from
Theorem~\ref{thm:reconstruction} then gives an exact equality of function
laws.

\subsection{Transport of the centered Green field}\label{subsec:transport}

For $\phi$ as in~\eqref{eq:automorphism}, the multiplier of parameter one
is $C_{1,\phi}(z)=(1-\bar az)/\sqrt{1-|a|^2}$. Direct calculation gives
\begin{equation}\label{eq:kernel-transformation}
 1-\phi(z)\overline{\phi(w)}
       =\frac{1-z\bar w}{C_{1,\phi}(z)\overline{C_{1,\phi}(w)}},
 \qquad |\phi'(z)|=|C_{1,\phi}(z)|^{-2}.
\end{equation}
In particular,
\begin{equation}\label{eq:multiplier-modulus}
 \log|C_{\alpha,\phi}(z)|=d_\alpha(\phi(z))-d_\alpha(z),
 \qquad (\log C_{\alpha,\phi})'(z)
                  =-\frac{\alpha\bar a}{1-\bar az}.
\end{equation}

\begin{lemma}\label{lem:green-transport}
If $\Xi\sim H_\alpha$, then $\eta=\phi^{-1}\Xi$ has law $H_\alpha$.
For every fixed $\phi$, on an event of probability one,
\begin{equation}\label{eq:green-transport}
 T_z(\phi^{-1}\Xi)=T_{\phi(z)}(\Xi),\qquad z\notin\phi^{-1}\Xi.
\end{equation}
The maps $T$ are those fixed in Proposition~\ref{prop:canonical}.
\end{lemma}
\begin{proof}
Use the planar kernel $\mathsf K_\alpha$ from
Lemma~\ref{lem:local-convergence}. By~\eqref{eq:kernel-transformation},
its kernel under pullback by $\phi$ is
\[
 |\phi'(z)|\mathsf K_\alpha(\phi(z),\phi(w))|\phi'(w)|
       =v_\phi(z)\mathsf K_\alpha(z,w)\overline{v_\phi(w)},
\]
where
$v_\phi(z)=\exp\{i(\alpha+1)\operatorname{Im}\Log C_{1,\phi}(z)\}$
has unit modulus and the logarithm is real at zero. Thus all correlation
determinants are unchanged. Their local Laplace expansion as
in~\eqref{eq:laplace-functional} proves the equality of point-process
laws. Formula~\eqref{eq:kernel-transformation} also shows that
$\nu_\alpha$ and the pseudohyperbolic distance are invariant.

For bounded compactly supported $f$, change variables in both the finite
sum and the mean to obtain, on this coupling,
\[
 \X_\eta(f)=\X_\Xi(f\circ\phi^{-1}).
\]
The pullback preserves the $L^2(\nu_\alpha)$ norm. The contraction in
Lemma~\ref{lem:variance} therefore extends this equality to every symbol
in that space. Pseudohyperbolic invariance gives
$q_z\circ\phi^{-1}=q_{\phi(z)}$ almost everywhere, proving
\eqref{eq:green-transport} for each fixed $z$ in the $L^2$ versions.

Take a countable dense set of centers, intersect the corresponding full
probability events, and also require that both $\Xi$ and $\eta$ belong
to $\mathscr G_\alpha$. They do so because their laws are $H_\alpha$.
The canonical fields are continuous off their respective configurations.
Hence the equality extends to all $z\notin\eta$ on this one event.
Its exceptional set may depend on $\phi$; neither the Borel version nor
the radial sequence has been changed.
\end{proof}

\begin{proposition}\label{prop:multiplier-cocycle}
With $a_\phi=\phi^{-1}(0)$ and
$\beta(\phi,\psi)=\Arg(1-\overline{a_\phi}\psi(0))\in(-\pi/2,\pi/2)$,
\begin{equation}\label{eq:main-projective}
 C_{\alpha,\phi\circ\psi}(z)
 =e^{-i\alpha\beta(\phi,\psi)}
   C_{\alpha,\phi}(\psi(z))C_{\alpha,\psi}(z).
\end{equation}
On the universal cover of $\Aut(\D)$, the continuously lifted logarithm
of the derivative defines the exact cocycle
\begin{equation}\label{eq:main-cover}
 \widetilde C_{\alpha,\widetilde\phi}(z)
       =\exp\{-\tfrac\alpha2\Log\phi'(z)\}.
\end{equation}
\end{proposition}

\subsection{Proof of covariance and the multiplier identities}
\label{subsec:covariance-proof}

\begin{proof}[Proof of Theorem~\ref{thm:covariance}(i) and
Proposition~\ref{prop:multiplier-cocycle}]
Fix $\phi$, and write $b=\phi(0)$ and $\eta=\phi^{-1}\Xi$.
On the full event in Lemma~\ref{lem:green-transport},
\eqref{eq:canonical-modulus} and~\eqref{eq:multiplier-modulus} show that
\[
 \frac{A_\alpha(\Xi)\circ\phi}{C_{\alpha,\phi}}
 \quad\text{and}\quad A_\alpha(\eta)
\]
have the same modulus off $\eta$ and the same zero divisor. Their
analytic quotient is therefore a constant of unit modulus. Because $b$
is deterministic, $P_\Xi(b)\ne0$ almost surely. Evaluation at zero
identifies the constant as $P_\Xi(b)/|P_\Xi(b)|$ and gives
\begin{equation}\label{eq:A-transport}
 \frac{A_\alpha(\Xi)(\phi(z))}{C_{\alpha,\phi}(z)}
   =\frac{P_\Xi(b)}{|P_\Xi(b)|}A_\alpha(\eta)(z).
\end{equation}
In particular, the separate radius and shape identities are
\begin{align}
 e^{c_\alpha-T_0(\eta)}
   &=\frac{|P_\Xi(b)|}{C_{\alpha,\phi}(0)}
                                  e^{c_\alpha-T_0(\Xi)},
                                   \label{eq:radius-transport}\\
 \frac{P_\Xi(\phi(z))}{P_\Xi(b)}
   &=\frac{C_{\alpha,\phi}(z)}{C_{\alpha,\phi}(0)}P_\eta(z).
                                   \label{eq:shape-transport}
\end{align}
They hold on the same coupling for this fixed automorphism; the second is
an identity of analytic functions on the disk.

Now represent $F_\alpha=UA_\alpha(\Xi)$ as in
Theorem~\ref{thm:reconstruction}. Conditional on $\Xi$, the random
variable $U P_\Xi(b)/|P_\Xi(b)|$ is still Haar. It is therefore
independent of $\eta$, whose law is $H_\alpha$. Applying the sampling
identity to~\eqref{eq:A-transport} proves
\eqref{eq:main-covariance}. The configuration-dependent phase has been
absorbed by conditional Haar averaging, not removed from a pathwise
identity.

For the composition formula, let
$L_\phi(z)=\Log C_{1,\phi}(z)$, normalized to be real at zero.
Apply~\eqref{eq:kernel-transformation} twice. The two analytic
multipliers $C_{1,\phi\circ\psi}$ and
$(C_{1,\phi}\circ\psi)C_{1,\psi}$ implement the same kernel
transformation, so their ratio is a constant of unit modulus.
Evaluation at zero determines its argument. Since
$1-\overline{a_\phi}\psi(0)$ has positive real part, this yields the
exact logarithmic identity
\begin{equation}\label{eq:log-cocycle}
 L_{\phi\circ\psi}(z)
       =L_\phi(\psi(z))+L_\psi(z)-i\beta(\phi,\psi).
\end{equation}
There is no residual multiple of $2\pi i$: at zero the right side is
real and equals the real logarithm of the positive normalized multiplier.
Multiplying~\eqref{eq:log-cocycle} by $\alpha$ and exponentiating proves
\eqref{eq:main-projective}.

On the universal cover, start the logarithm of $\phi'$ at zero for the
identity map and continue it along the lifted path. The derivative chain
rule has an exact logarithmic lift,
\[
 \Log(\phi\circ\psi)'=(\Log\phi')\circ\psi+\Log\psi',
\]
with the composition of lifts. Formula~\eqref{eq:main-cover} therefore
satisfies the exact multiplicative cocycle relation. Along the positive
central generator, corresponding to a full rotation, $\Log\phi'$
increases by $2\pi i$, so this cocycle is multiplied by
$e^{-i\pi\alpha}$. It descends to $\Aut(\D)$ precisely when $\alpha$
is a positive even integer; if $\alpha=2n$, it is $(\phi')^{-n}$.
The derivative-normalized cocycle need not be the positive-at-zero
section~\eqref{eq:multiplier}; they differ by a spatially constant phase.
\end{proof}

Differentiating~\eqref{eq:shape-transport} gives the corresponding
canonical logarithmic-derivative transport,
\begin{equation}\label{eq:connection-transport}
 S_{\phi^{-1}\Xi}(z)
       =\phi'(z)S_\Xi(\phi(z))-(\log C_{\alpha,\phi})'(z).
\end{equation}
This holds as a meromorphic identity on the whole disk. It is a
consequence of the analytic shape identity and does not require an
additional choice of cutoff or version.

\subsection{Point laws and logarithmic derivatives}
\label{subsec:point-laws}

The covariance theorem determines all one-point distributions and gives
explicit transformations of finite jets. We record these direct
consequences to describe the normalization at points other than zero.

\begin{corollary}\label{cor:point-laws}
For every deterministic $b\in\D$,
$(1-|b|^2)^{\alpha/2}F_\alpha(b)\overset d=F_\alpha(0)$. In particular,
\begin{equation}\label{eq:point-mellin}
 \E|F_\alpha(b)|^{2u}=(1-|b|^2)^{-\alpha u}M_\alpha(u),
 \qquad \operatorname{Re}u>-1,
\end{equation}
with the same maximal absolute moment domain. Moreover,
\begin{equation}\label{eq:covariance-kernel}
 \E[F_\alpha(z)\overline{F_\alpha(w)}]
       =\Gamma(1+\alpha)(1-z\bar w)^{-\alpha},
 \qquad \E[F_\alpha(z)F_\alpha(w)]=0.
\end{equation}
\end{corollary}
\begin{proof}
Choose $\phi(0)=b$ and evaluate~\eqref{eq:main-covariance} at zero.
Here $|a|=|b|$ and
$C_{\alpha,\phi}(0)=(1-|b|^2)^{-\alpha/2}>0$, which proves the
one-point law and~\eqref{eq:point-mellin}. Deterministic scaling preserves
the absolute moment domain.

The Barnes recursion gives $M_\alpha(1)=\Gamma(1+\alpha)$, and thus the
diagonal of the first kernel in~\eqref{eq:covariance-kernel}. The mean
inequality for $|F_\alpha|^2$ and~\eqref{eq:point-mellin} imply
$\E\sup_{D_r}|F_\alpha|^2<\infty$ for every $r<1$. Consequently
expectation can be interchanged with local Cauchy integrals, making that
kernel holomorphic in its first variable and antiholomorphic in its
second. Such a kernel is determined by its diagonal: expand on a small
disk in $z^m\bar z^n$ and identify the coefficients, then continue
analytically. This proves the first equality. The second is integrable by
Cauchy--Schwarz and vanishes by the independent Haar phase.
\end{proof}

For a deterministic $b$ define
$L_r^{(b)}(F)=(\log F)^{(r)}(b)/(r-1)!$ whenever $F(b)\ne0$.
The preceding corollary ensures this for $F_\alpha$ almost surely.

\begin{corollary}\label{cor:moving-jets}
Let $\phi$ be as in~\eqref{eq:automorphism}, with $b=\phi(0)$ and
$d=\phi'(0)$. For every $K\ge1$,
\begin{equation}\label{eq:moving-jets}
 \left((1-|a|^2)^{\alpha/2}F_\alpha(b), 
   \left\{\sum_{r=1}^k\binom kr d^r\bar a^{k-r}L_r^{(b)}(F_\alpha)
                   +\alpha\bar a^k\right\}_{k=1}^K\right)
 \overset d=(F_\alpha(0),L_1(F_\alpha),\ldots,L_K(F_\alpha)).
\end{equation}
\end{corollary}
\begin{proof}
Since $\phi(z)=b+dz/(1-\bar az)$, expansion of a local logarithm $H$ at
$b$ gives
\[
 k[z^k]H(\phi(z))
  =\sum_{r=1}^k\frac{k}{r}\binom{k-1}{r-1}
                   d^r\bar a^{k-r}L_r^{(b)}(F).
\]
Use $(k/r)\binom{k-1}{r-1}=\binom kr$ and
$L_k(C_{\alpha,\phi})=-\alpha\bar a^k$, and apply the basepoint--jet
map to $F_\alpha\circ\phi/C_{\alpha,\phi}\overset d=F_\alpha$.
\end{proof}

\section{Classification of covariant realizations}\label{sec:geometric}

We prove parts~(ii)--(iii) of Theorem~\ref{thm:geometric}. The comparison
with the canonical function leaves a spatially constant random factor.
Two elementary consequences of invariance identify its dependence on the
zeros and its phase.

\subsection{Mixing of the zeros and circular symmetry}
\label{subsec:scale-inputs}

For the standard Bergman process $H_1$, strong mixing under every
hyperbolic or parabolic disk automorphism was proved by Bufetov, Qiu
and Shamov~\cite[Lemma~8.5]{BQS2021}. The following direct
correlation-kernel argument gives the corresponding statement for
every $H_\alpha$, $\alpha>0$.

\begin{lemma}
\label{lem:hyperbolic-mixing}
Let $g(z)=(z+\tanh1)/(1+z\tanh1)$. If $\Xi\sim H_\alpha$ and
$V,W\in L^2(H_\alpha)$, then
\begin{equation}\label{eq:hyperbolic-mixing}
 \E[V(\Xi)W(g^n\Xi)]\longrightarrow\E V(\Xi)\,\E W(\Xi)
 \qquad(n\to\infty).
\end{equation}
In particular, the action of $g$ is ergodic.
\end{lemma}
\begin{proof}
Relative to the invariant measure
$dA_{\mathrm h}=dm/[\pi(1-|z|^2)^2]$, the process has kernel
\[
 \mathcal K_\alpha(z,w)=
 \alpha\frac{(1-|z|^2)^{(\alpha+1)/2}(1-|w|^2)^{(\alpha+1)/2}}
                  {(1-z\bar w)^{\alpha+1}}.
\]
Its diagonal is $\alpha$, and
\begin{equation}\label{eq:mixing-kernel-decay}
 |\mathcal K_\alpha(z,w)|
 =\alpha(1-\delta(z,w)^2)^{(\alpha+1)/2},\qquad
 \delta(z,w)=\left|\frac{z-w}{1-z\bar w}\right|.
\end{equation}
For compact sets $A,B\Subset\D$, $g^{\pm n}B$ approaches the boundary
uniformly. Consequently $\delta(z,g^{-n}w)\to1$ uniformly on $A\times B$,
and these two sets are disjoint for all sufficiently large $n$.

First take $V(\xi)=e^{-\langle v,\xi\rangle}$ and
$W(\xi)=e^{-\langle w,\xi\rangle}$, where $v,w\ge0$ are continuous
and supported in $A,B$. Expand their joint Laplace functional as in
\eqref{eq:laplace-functional}, separating the points in $A$ and
$g^{-n}B$, and pull the latter variables back to $B$. In the correlation
determinant, the cross-block entries tend uniformly to zero by
\eqref{eq:mixing-kernel-decay}. The determinant of the second diagonal
block is unchanged by this pullback, by invariance of the correlation
measures and of $A_{\mathrm h}$. Thus, for each pair of orders $k,l$,
the integrand tends to the product of the two correlation determinants.
Hadamard's inequality bounds the absolute value of that term by
\[
 \frac{\alpha^{k+l}A_{\mathrm h}(A)^kA_{\mathrm h}(B)^l}{k!\,l!}.
\]
This is summable over $k,l$, so the joint Laplace functional tends to
the product of the separate ones.

The linear span of these local Laplace functionals is an algebra
containing constants and generating the configuration sigma-field; it is
dense in $L^2(H_\alpha)$ by the monotone class theorem. Invariance of
$H_\alpha$, proved in Lemma~\ref{lem:green-transport}, makes composition
by $g^n$ an $L^2$ isometry. Approximation and Cauchy--Schwarz extend the
limit to $V,W\in L^2$. Applying it to an invariant $V$ and its complex
conjugate shows that $V$ is almost surely constant, proving ergodicity.
\end{proof}

\begin{lemma}
\label{lem:forced-phase}
Suppose $F(0)\ne0$ almost surely and
$F\circ\phi\overset d=C_{\alpha,\phi}F$ for every $\phi\in\Aut(\D)$,
with $\alpha>0$ and the convention~\eqref{eq:multiplier}. Then
$e^{i\theta}F\overset d=F$ for every $\theta\in\R$.
If $Y=Y(F)$ is any Borel statistic unchanged by constant phases, then
$F(0)/|F(0)|$ is Haar and independent of $Y$.
\end{lemma}
\begin{proof}
Put $\mathcal T_\phi f=f\circ\phi/C_{\alpha,\phi}$. The multiplier
identity~\eqref{eq:main-projective} gives
\[
 \mathcal T_\psi\mathcal T_\phi
       =e^{-i\alpha\beta(\phi,\psi)}\mathcal T_{\phi\circ\psi}.
\]
All three transformation operators preserve the law of $F$, so
$e^{-i\alpha\beta(\phi,\psi)}F\overset d=F$. Fix
$a_\phi=r\in(0,1)$ and choose $\psi$ with $\psi(0)=it$, $-1<t<1$.
Then $\beta(\phi,\psi)=\Arg(1-irt)$ ranges over an interval about zero.
Because $\alpha>0$, the corresponding phases generate $\T$. This proves
circular symmetry. With $U=F(0)/|F(0)|$, it follows that
$(U,Y)\overset d=(e^{i\theta}U,Y)$. Integrating bounded Borel test
functions against $d\theta/(2\pi)$ proves Haar distribution and
independence, exactly as in~\eqref{eq:haar-average}.
\end{proof}

\subsection{The independent scale and its normalization}
\label{subsec:independent-scale-proof}

\begin{proof}[Proof of Theorem~\ref{thm:geometric}(ii)--(iii)]
Let $F$ satisfy the hypotheses and covariance in part~(ii), and put
$\Xi=Z(F)$. The functions $F$ and $A_\alpha(\Xi)$ have the same divisor,
so their quotient extends to a zero-free analytic function. Define
\begin{equation}\label{eq:harmonic-comparison}
 h(z)=\log\left|\frac{F(z)}{A_\alpha(\Xi)(z)}\right|.
\end{equation}
This is a real harmonic field. At a common zero use the extended
quotient: if $a\in\Hol\setminus\{0\}$, then
$m(a,z)=\min\{j\ge0:a^{(j)}(z)\ne0\}$ is jointly Borel in $(a,z)$,
and the quotient there is
$f^{(m(a,z))}(z)/a^{(m(a,z))}(z)$. This proves joint Borel measurability
of the extended quotient and of $h$. Outside the full canonical event
they may be assigned fixed values.

For each fixed $\phi$, covariance and the zero-divisor map give
\[
 (\mathcal T_\phi F,\phi^{-1}\Xi)\overset d=(F,\Xi).
\]
By~\eqref{eq:A-transport}, the logarithmic modulus of the corresponding
canonical quotient is $h\circ\phi$. Hence $h\circ\phi\overset d=h$.
The assumed $L^2$ integrability and
$\log A_\alpha(\Xi)(0)=c_\alpha-T_0(\Xi)$ imply $h(0)\in L^2$.
Transitivity gives
\begin{equation}\label{eq:h-stationary-second}
 \E h(z)^2=\E h(0)^2<\infty,\qquad z\in\D.
\end{equation}
For each $r<1$, the pathwise harmonic mean identity and Fubini yield
\[
 \E\frac1{2\pi}\int_0^{2\pi}(h(re^{it})-h(0))^2dt
 =\frac1{2\pi}\int_0^{2\pi}\E h(re^{it})^2dt-\E h(0)^2=0.
\]
Continuity, followed by a countable dense set of radii, shows that
\begin{equation}\label{eq:random-scale}
 h(z)=B\quad(z\in\D),\qquad
 B=\log|F(0)|-c_\alpha+T_0(\Xi)\in L^2.
\end{equation}
The analytic quotient now has constant modulus. Thus, with
$U=F(0)/|F(0)|$,
\begin{equation}\label{eq:scale-on-divisor}
 F=e^B U A_\alpha(\Xi)\quad\text{almost surely}.
\end{equation}

We next prove independence, rather than assume it. The transformation
of the comparison field shows that $B(\mathcal T_\phi F)=B(F)$ almost
surely for each fixed $\phi$. Consequently
\begin{equation}\label{eq:scale-joint-invariance}
 (B,\phi^{-1}\Xi)\overset d=(B,\Xi).
\end{equation}
For a bounded Borel function $a$ on $\R$, choose a conditional expectation
$b(\xi)=\E[a(B)\mid\Xi=\xi]$. Applying~\eqref{eq:scale-joint-invariance}
with $\phi=g$, for every bounded Borel $f$,
\[
 \int b(\xi)f(\xi)\,dH_\alpha(\xi)
 =\int b(\xi)f(g^{-1}\xi)\,dH_\alpha(\xi)
 =\int b(g\xi)f(\xi)\,dH_\alpha(\xi).
\]
Thus $b\circ g=b$ almost surely. Ergodicity from
Lemma~\ref{lem:hyperbolic-mixing} makes $b$ constant, equal to $\E a(B)$.
Taking indicators of Borel sets proves that $B$ and $\Xi$ are independent.
There is no need to choose invariant versions for an uncountable family
of automorphisms.

The pair $(B,\Xi)$ is a Borel function of $F$ unchanged by constant phases.
Lemma~\ref{lem:forced-phase} makes $U$ Haar and independent of this pair.
It follows that $B$, $U$ and $\Xi$ are mutually independent. By
Theorem~\ref{thm:reconstruction}(ii), $U A_\alpha(\Xi)$ has law
$F_\alpha$.  Setting $S=e^B$ proves
\eqref{eq:independent-scale-classification}; moreover,
\eqref{eq:random-scale} shows that the representing scale automatically
satisfies $\log S\in L^2$.

Conversely, suppose that
$F\overset d=S F_\alpha$, where $S>0$ is independent of $F_\alpha$.
Since $F$ satisfies the standing assumption
$\E(\log|F(0)|)^2<\infty$ and
$\log|F_\alpha(0)|\in L^2$, the identity
\[
 \log|S F_\alpha(0)|
   =\log S+\log|F_\alpha(0)|
\]
shows that $\log S\in L^2$. Multiplication by the independent positive
constant $S$ preserves both the divisor law and the covariance of
$F_\alpha$. This proves the equivalence in part~(ii).
The expression for $B$ in~\eqref{eq:random-scale} is a fixed Borel
statistic of $F$, and for the representation $S F_\alpha$ it equals
$\log S$ almost surely. Hence the law of $S$ is unique.

Finally, independence gives
\begin{equation}\label{eq:scale-log-moments}
 \E\log|F(0)|=c_\alpha+\E\log S,\qquad
 \Var(\log|F(0)|)=v_\alpha+\Var(\log S).
\end{equation}
The canonical logarithmic moments follow by differentiating
\eqref{eq:main-barnes}, justified by the bilateral exponential moments
near zero in Proposition~\ref{prop:scalar}. The variance is minimal
precisely when $S$ is deterministic. The additional mean normalization
then gives $S=1$, proving part~(iii).
\end{proof}

In particular, prescribing $\E|F(0)|^{2u}=M_\alpha(u)$ on a real
neighborhood of zero gives uniqueness within the covariant class with
divisor law $H_\alpha$: these bilateral moments imply
\eqref{eq:geometric-normalization}. Appendix~\ref{sec:direct-normalization}
gives a direct proof of the normalized conclusion by averaging the Green
field. The stronger assertion that all scales are independent uses mixing.

\section{Continuity and the integer models}
\label{sec:continuity-integer}

The fixed-parameter construction is now complete. We first vary the
parameter using the same Green estimates, and then identify the integer
members of the family by geometric uniqueness. These are separate uses
of the constructed laws; neither is obtained by an exchange of the two
limits $N\to\infty$ and $\alpha_j\to\alpha$.

\subsection{Proof of Theorem~\ref{thm:continuity}(i)}
\label{subsec:continuity-proof}

\begin{proof}
Let $\alpha_j\to\alpha>0$, and restrict these parameters to
$[a,b]\Subset(0,\infty)$. The constants $c_\beta$ are bounded on this
interval, $I_\beta\le I_b$, and the Green variance bound is at most
$b\pi^2/12$. The canonical modulus formula and the sampling law therefore
give, for every $s<1$,
\[
 \sup_{\beta\in[a,b]}\sup_{|z|\le s}
              \E\log^+|F_\beta(z)|<\infty.
\]
The subharmonic mean argument of
Proposition~\ref{prop:finite-tightness} supplies compact-open tightness for
this family of infinite-volume functions. By
Lemma~\ref{lem:local-convergence}, the divisors
$\Xi_j=Z(F_{\alpha_j})$ converge in law to $H_\alpha$.

Take any joint subsequential limit
$(F_{\alpha_j},\Xi_j)\Longrightarrow(F,\Xi)$.
For fixed centers $z_1,\ldots,z_m$, use common continuous compactly
supported approximations $f_{i,n}$ of $q_{z_i}$ in $L^2(\nu_1)$.
The errors under every parameter in $[a,b]$ are bounded by
\[
 b\sum_{i=1}^m\norm{q_{z_i}-f_{i,n}}_{L^2(\nu_1)}^2,
\]
which tends to zero. For a fixed approximation the centering is
$\alpha_j\int f_{i,n}\,d\nu_1$, which converges to its value at
$\alpha$. Vague convergence and
Lemma~\ref{lem:joint-approximation} therefore give
\begin{equation}\label{eq:parameter-joint-green}
 (F_{\alpha_j},\Xi_j,T_{z_1,\alpha_j},\ldots,T_{z_m,\alpha_j})
 \Longrightarrow(F,\Xi,T_{z_1,\alpha},\ldots,T_{z_m,\alpha}).
\end{equation}
The parameter subscript on $T$ here indicates its intensity and its
canonical version.

The constants $d_{\alpha_j}(z)$ tend to $d_\alpha(z)$. Passing the
exponential modulus relation to~\eqref{eq:parameter-joint-green} shows
that $|F(z)|=\exp\{d_\alpha(z)-T_{z,\alpha}(\Xi)\}$ for each
fixed $z$. At zero this excludes the zero function, and continuity of the
zero-divisor map then gives $Z(F)=\Xi$. Comparison on a countable dense
set with $A_\alpha(\Xi)$ yields $F=UA_\alpha(\Xi)$ for some
$U\in\T$.

Each joint law $(F_{\alpha_j},\Xi_j)$ is unchanged by multiplying its
function coordinate by $e^{i\theta}$, by
Theorem~\ref{thm:reconstruction}. This continuous action preserves the
same equality in the limit. The Haar averaging
in~\eqref{eq:haar-average} again makes $U$ independent of $\Xi$.
Thus every joint cluster law is the canonical law at $\alpha$.
Tightness proves $F_{\alpha_j}\Longrightarrow F_\alpha$ in $\Hol$.
The argument does not require continuity of the Borel reconstruction map.
\end{proof}

\subsection{Proof of Theorem~\ref{thm:integer}(ii)}
\label{subsec:integer-proof}
The integer function-level limit was already established in  the proof of
\cite[Theorem~4]{Krishnapur2009} via the truncated-unitary approximation. For completeness, and to show
how the integer models fit the geometric characterization developed
here, we give an alternative identification. In the argument below,
the only external input we use from Krishnapur is the zero-divisor
identity: for independent standard complex Ginibre $m\times m$
coefficients, the zero divisor of their analytic matrix determinant
has law $H_m$, with the measure and kernel in~\eqref{eq:model}.

\begin{proof}
Let $M(z)=\sum_{n\ge0}G_nz^n$ and $D_m(z)=\det M(z)$.
For each $r<1$, $\E\sum_n\norm{G_n}r^n<\infty$, for any matrix norm.
Taking a countable sequence of $r$ tending to one proves almost sure local
absolute convergence of the matrix series. Hence $D_m\in\Hol$.
Absolute continuity of $G_0$ implies $D_m(0)=\det G_0\ne0$ almost surely.
The cited theorem gives $Z(D_m)\sim H_m$.

Every entry of $M$ is a proper centered Gaussian analytic function with
covariance $(1-z\bar w)^{-1}$ and vanishing pseudocovariance. By
\eqref{eq:kernel-transformation}, composition with $\phi$ gives the
same covariance as multiplication by $C_{1,\phi}$. Proper centered
Gaussian finite vectors are determined by these covariances, so the
corresponding entrywise finite-dimensional laws agree. Distinct entries
remain independent. Countable evaluation determines the analytic-function
law, and therefore
\[
 M\circ\phi\overset d=C_{1,\phi}M
 \quad\text{as matrix-valued analytic functions}.
\]
Taking determinants yields
$D_m\circ\phi\overset d=C_{1,\phi}^mD_m=C_{m,\phi}D_m$.

For the basepoint moments, apply Gram--Schmidt to the columns of $G_0$.
The squared norm of the $j$th residual, conditionally on the preceding
columns, has the gamma law of shape $m-j+1$ and scale one. Its conditional
law is independent of those columns by isotropy of the standard complex
Gaussian vector. The successive residual norms are consequently
independent, and
\[
 |\det G_0|^2\overset d=\prod_{j=1}^m\gamma_j,
 \qquad \gamma_j\sim\operatorname{Gamma}(j,1)\text{ independent}.
\]
It follows that
\begin{equation}\label{eq:ginibre-mellin}
 \E|D_m(0)|^{2u}
      =\prod_{j=1}^m\frac{\Gamma(j+u)}{\Gamma(j)}
      =\frac{G(m+1+u)}{G(m+1)G(1+u)}=M_m(u),
 \quad\operatorname{Re}u>-1.
\end{equation}

Finally, multiplying every coefficient matrix by $e^{i\theta/m}$ leaves
their joint law unchanged and multiplies $D_m$ by $e^{i\theta}$.
Its basepoint modulus and all its normalized logarithmic jets remain
unchanged, so Haar averaging also verifies their independence from the
phase directly. For the characterization it suffices to differentiate
\eqref{eq:ginibre-mellin} at zero: the bilateral exponential moments of
$\log|D_m(0)|$ justify the differentiation and give mean $c_m$ and
variance $v_m$. Theorem~\ref{thm:geometric}(iii) therefore identifies
$D_m\overset d=F_m$ in the whole space $\Hol$.
\end{proof}

The finite matrix model gives a direct interpretation of the same
normalization. At integer parameters, the following compact-open
function limit was already obtained in Krishnapur's proof of
Theorem~4
\cite[Section~5]{Krishnapur2009}. We record it here as a corollary of
the present continuous-parameter construction. The finite eigenvalue
law is the truncated-unitary law of~\cite{ZS2000}, stated as
Result~12 in~\cite{Krishnapur2009}.

\begin{corollary}\label{cor:unitary}
Let $V_N$ be the $N\times N$ principal truncation of a Haar-distributed
$U(N+m)$ matrix, with $m\ge1$ fixed. Then
\begin{equation}\label{eq:unitary-limit}
 N^{m/2}\frac{\det(zI_N-V_N)}{\det(I_N-zV_N^*)}
       \Longrightarrow \det\Bigl(\sum_{n\ge0}G_nz^n\Bigr)
       \quad\text{in }\Hol.
\end{equation}
\end{corollary}
\begin{proof}
The eigenvalues of $V_N$ have exactly the law $\Xi_{N,m}$ and lie in
$\D$ almost surely. Counting algebraic multiplicities, the eigenvalues
of $V_N^*$ are their complex conjugates, whether or not $V_N$ is normal.
Consequently the determinant ratio in~\eqref{eq:unitary-limit}, including
its factor $N^{m/2}$, is precisely the Blaschke product
$F_{N,m}$ in distribution as a function. Apply
Theorem~\ref{thm:realization} and then Theorem~\ref{thm:integer}(ii).
\end{proof}

\section{The Mellin--Fourier characterization}\label{sec:transforms}

The function law has already been constructed without using its joint
transforms. We now identify those transforms by changing measure at the
basepoint. The tilted point processes and their logarithmic jets are
first defined independently of the function. This also explains why no
polynomial moment assumption on the jets is needed in the characterization.

Our second characterization does not assume a zero law or covariance. Its
right-hand side is specified by an independently defined tilted
projection process. For real $u>-1$, put, for $z\ne0$,
\begin{equation}\label{eq:tilted-basis}
 \varphi_{n,\alpha,u}(z)=
 \frac{|z|^u(1-|z|^2)^{(\alpha-1)/2}z^n}
      {\sqrt{\pi B(n+u+1,\alpha)}},\qquad n\ge0,
\end{equation}
where $B$ is the beta function. The value at zero can be set arbitrarily.
These functions are orthonormal in $L^2(dm)$. Let $\Prob_{\alpha,u}$ be the
projection determinantal law with planar kernel
\begin{equation}\label{eq:tilted-kernel}
 \mathsf K_{\alpha,u}(z,w)=
       \sum_{n=0}^\infty\varphi_{n,\alpha,u}(z)
                           \overline{\varphi_{n,\alpha,u}(w)}.
\end{equation}
Thus $\Prob_{\alpha,0}=H_\alpha$. For $\Xi_{\alpha,u}\sim\Prob_{\alpha,u}$,
set $h_k(w)=\bar w^k-w^{-k}$ for $w\ne0$, and define
\begin{equation}\label{eq:dpp-jets}
 L_{k,\alpha,u}=
 \sum_{\substack{x\in\Xi_{\alpha,u}\\|x|\le r_0}}h_k(x)
 + L^2\!\!\lim_{R\uparrow1}
       \sum_{\substack{x\in\Xi_{\alpha,u}\\r_0<|x|<R}}h_k(x),
 \qquad 0<r_0<1.
\end{equation}
Only the exterior term is an $L^2$ limit. The interior sum is finite almost
surely. Section~\ref{subsec:tilted-statistics} proves existence and
independence of $r_0$, as well as the equivalent centered-cutoff definition.
Define its characteristic function by
\begin{equation}\label{eq:psi}
 \Psi_{\alpha,u}^{(K)}(\mathbf t)=
 \E_{\alpha,u}\exp\left\{i\operatorname{Re}
                 \sum_{k=1}^K\bar t_kL_{k,\alpha,u}\right\},
 \qquad \mathbf t\in\C^K.
\end{equation}

\begin{theorem}[Mellin--Fourier characterization]\label{thm:transform}
The law of $F_\alpha$ is the unique law of a random $F\in\Hol$ with
$F(0)\ne0$ almost surely for which there is an $\varepsilon\in(0,1)$ such
that, for every real $|u|<\varepsilon$, $q\in\Z$, $K\ge1$ and
$\mathbf t\in\C^K$, the following expectation is well defined and
\begin{equation}\label{eq:main-transform}
 \E\left[|F(0)|^{2u}
       \left(\frac{F(0)}{|F(0)|}\right)^q
       \exp\left\{i\operatorname{Re}\sum_{k=1}^K\bar t_kL_k(F)\right\}\right]
=\ind_{\{q=0\}}M_\alpha(u)
                          \Psi_{\alpha,u}^{(K)}(\mathbf t).
\end{equation}
The law of $F_\alpha$ satisfies this identity for every real $u>-1$.
\end{theorem}

Only real tilts appear in the prescribed transform. In the uniqueness
proof, analytic continuation is applied to the transforms of two candidate
laws, not to a presumed complex extension of
$\Psi_{\alpha,u}^{(K)}$.

\subsection{Tilted projections and centered logarithmic jets}
\label{subsec:tilted-statistics}

Fix $\alpha>0$ and a real $u>-1$. Recall the planar orthonormal functions
$\varphi_{n,\alpha,u}$ in~\eqref{eq:tilted-basis}. Their first $N$ terms
form a projection kernel, denoted by $\mathsf K_{N,\alpha,u}$. The
following bounds handle both the singularity allowed at the origin and
the accumulation of points near the boundary.

\begin{lemma}\label{lem:tilted-projections}
The kernel $\mathsf K_{\alpha,u}$ is a locally trace-class orthogonal
projection. The laws of the projections $\mathsf K_{N,\alpha,u}$ converge
weakly on $\Conf$ to $\Prob_{\alpha,u}$. Their intensities satisfy
\begin{equation}\label{eq:tilted-intensity}
 \begin{split}
  \rho_{N,\alpha,u}(z)&\le\rho_{\alpha,u}(z),\\
  \rho_{\alpha,u}(z)
    &=\frac{|z|^{2u}(1-|z|^2)^{\alpha-1}}{\pi}
         \sum_{n=0}^\infty\frac{|z|^{2n}}{B(n+u+1,\alpha)}
      \le C_{\alpha,J}\frac{|z|^{2u}}{(1-|z|^2)^2},
 \end{split}
\end{equation}
uniformly for $u$ in a compact interval $J\subset(-1,\infty)$ and
$z\ne0$. Moreover, the finite projection law is
\begin{equation}\label{eq:finite-tilt}
 \frac{dQ_{N,u}}{d\Law(\Xi_{N,\alpha})}(\xi)
     =\frac{X_N(\xi)^u}{M_{N,\alpha}(u)},
 \qquad X_N(\xi)=N^\alpha\prod_{x\in\xi}|x|^2.
\end{equation}
\end{lemma}
\begin{proof}
Polar integration proves orthonormality. The sum of the rank-one
projections is the orthogonal projection onto their closed span. The
beta-function ratios obey
\[
 \frac{B(n+1,\alpha)}{B(n+u+1,\alpha)}=1+O_{\alpha,J}((n+1)^{-1}).
\]
They are positive and bounded above and below, uniformly in $n\ge0$ and
$u\in J$; for finitely many $n$ use continuity and positivity of the beta
function. Comparing the diagonal series with its $u=0$ counterpart gives
\eqref{eq:tilted-intensity}. In particular the intensity is integrable on
every compact subset of the disk, including at the origin since $u>-1$.
This proves local trace class.

For fixed $u$, the restricted finite and infinite projections have a
positive difference, whose trace on a compact set is the integral of
$\rho_{\alpha,u}-\rho_{N,\alpha,u}$. That integral tends to zero by
monotone convergence. Alternatively, the kernel series converges locally
uniformly away from zero, and the determinant of each local correlation
matrix is bounded by the product of the intensities. Near zero this
product is integrable by~\eqref{eq:tilted-intensity}. The Laplace series
argument of Lemma~\ref{lem:local-convergence} therefore applies on compact
sets containing zero as well. It proves convergence on $\Conf$; no
uniform bound on the kernel at zero is being asserted for negative $u$.

Finally, multiply the density~\eqref{eq:finite-density} by
$\prod_j|z_j|^{2u}$ and renormalize. Orthogonality of monomials gives the
density
\[
 \frac{|\Delta(z_1,\ldots,z_N)|^2}
      {N!\prod_{n=0}^{N-1}\pi B(n+u+1,\alpha)}
 \prod_{j=1}^N |z_j|^{2u}(1-|z_j|^2)^{\alpha-1}
\]
with respect to $dm(z_1)\cdots dm(z_N)$. It is the density of the projection
$\mathsf K_{N,\alpha,u}$. The factors $N^{\alpha u}$ cancel in the
normalization, and Proposition~\ref{prop:scalar} identifies that
normalization as $M_{N,\alpha}(u)$. This gives~\eqref{eq:finite-tilt}.
\end{proof}

The only possible failure of jet moments comes from points near zero.
Boundary regularization is a different issue and is controlled by a
square-integrable exterior symbol.

\begin{lemma}\label{lem:tilted-jets}
The variables in~\eqref{eq:dpp-jets} are well defined, jointly for every
finite number of indices, and do not depend on $r_0$. For each $k\ge1$
and $r_0<R<S<1$,
\begin{equation}\label{eq:tilted-jet-tail}
 \E_{\alpha,u}\left|
       \sum_{\substack{x\in\Xi_{\alpha,u}\\R<|x|<S}}h_k(x)
                         \right|^2
       \le C_{\alpha,u,k,r_0}(S^2-R^2).
\end{equation}
Along $R_n=1-2^{-4n}$, the radial sums converge almost surely to
$L_{k,\alpha,u}$, simultaneously for all $k$.

More generally, let $\chi_n$ be bounded, compactly supported measurable
cutoffs and suppose
\begin{equation}\label{eq:tilted-cutoff-condition}
 \int_{|w|>r_0}|1-\chi_n(w)|^2|h_k(w)|^2
                              \rho_{\alpha,u}(w)\,dm(w)\longrightarrow0.
\end{equation}
The centered exterior statistics of
$\chi_n h_k\ind_{\{|w|>r_0\}}$ then have the same $L^2$ limit as the
exterior term in~\eqref{eq:dpp-jets}.
\end{lemma}
\begin{proof}
On $|w|\ge r_0$, the elementary bound used for
\eqref{eq:analytic-symbol-bound} gives
$|h_k(w)|\le C_{k,r_0}(1-|w|^2)$. By
\eqref{eq:tilted-intensity}, this exterior symbol belongs to
$L^2(\rho_{\alpha,u}\,dm)$. Its integral over every radial annulus
vanishes: the intensity is radial and both angular modes of $h_k$ have
nonzero frequency. Lemma~\ref{lem:variance}, applied to the tilted
projection itself, proves~\eqref{eq:tilted-jet-tail} and the radial
$L^2$ Cauchy property.

The process is locally finite and has no point at zero, since its
intensity is absolutely continuous. The interior term in
\eqref{eq:dpp-jets} is consequently a finite sum almost surely, without
any assertion that it is square integrable. Changing $r_0$ moves a finite
annulus from one term to the other; its centered mean is zero. This proves
independence of the split. Letting $S\uparrow1$ in the tail estimate
bounds the squared error at $R_n$ by $C(1-R_n^2)$. These errors are
summable, so Chebyshev's inequality and Borel--Cantelli give almost sure
convergence, first for each $k$ and then simultaneously for all $k$.
Finally, the contraction~\eqref{eq:variance-bound} applied to the
$L^2$ error in~\eqref{eq:tilted-cutoff-condition} proves the last assertion.
Thus nonradial cutoffs are allowed, provided their exterior statistics
are correctly centered.
\end{proof}

\subsection{The basepoint change of measure}\label{subsec:change-measure}

We next compare the tilted law just constructed with $H_\alpha$. The
normalizing statistic in this comparison is the canonical $T_0$ centered
with respect to the \emph{untilted} intensity $\nu_\alpha$.

\begin{proposition}\label{prop:tilted-law}
For every real $u>-1$,
\begin{equation}\label{eq:tilted-RN}
 \frac{d\Prob_{\alpha,u}}{dH_\alpha}(\xi)
       =\frac{\exp\{2u(c_\alpha-T_0(\xi))\}}{M_\alpha(u)}.
\end{equation}
These two configuration laws are equivalent. On the canonical full
measure set, the tilted jets satisfy
\begin{equation}\label{eq:tilted-canonical-jets}
       L_{k,\alpha,u}=L_k(P_\Xi),\qquad k\ge1.
\end{equation}
Furthermore, if the pair $(F_{N,\alpha},\Xi_{N,\alpha})$ is given the
weight $|F_{N,\alpha}(0)|^{2u}/M_{N,\alpha}(u)$, its law converges weakly
to the law of $(F_\alpha,Z(F_\alpha))$ weighted by
$|F_\alpha(0)|^{2u}/M_\alpha(u)$.
\end{proposition}
\begin{proof}
Choose $p>1$ such that $pu>-1$. Proposition~\ref{prop:scalar} gives
\[
 \sup_N\E|F_{N,\alpha}(0)|^{2pu}<\infty.
\]
Thus the weights $|F_{N,\alpha}(0)|^{2u}$ are uniformly integrable,
including when $u<0$. By Theorem~\ref{thm:realization}, the unweighted
pairs converge jointly, and their limiting function is nonzero at zero.
The weight is continuous at every such function. Truncation of the weight,
followed by uniform integrability, therefore gives, for bounded continuous
$B$ on $\Hol\times\Conf$,
\begin{equation}\label{eq:weighted-joint-limit}
 \frac{\E\bigl[|F_{N,\alpha}(0)|^{2u}
                      B(F_{N,\alpha},\Xi_{N,\alpha})\bigr]}
                 {M_{N,\alpha}(u)}\longrightarrow
   \frac{\E\bigl[|F_\alpha(0)|^{2u}
                      B(F_\alpha,Z(F_\alpha))\bigr]}{M_\alpha(u)}.
\end{equation}
This proves the weighted joint convergence.

For a test function depending only on the configuration, the law on the
left is $Q_{N,u}$ by~\eqref{eq:finite-tilt}; its limit is
$\Prob_{\alpha,u}$ by Lemma~\ref{lem:tilted-projections}. On the right,
Theorem~\ref{thm:reconstruction} gives
$|F_\alpha(0)|^{2u}=\exp\{2u(c_\alpha-T_0(\Xi))\}$ on the same
configuration $\Xi$. Equality on bounded continuous configuration tests,
which determine probability measures, proves~\eqref{eq:tilted-RN}.
The density is finite and strictly positive $H_\alpha$-almost surely,
so the two laws are equivalent.

The canonical set in Proposition~\ref{prop:canonical} consequently has
full $\Prob_{\alpha,u}$ measure. On it the radial sums along $R_n$
converge to $L_k(P_\Xi)$ by~\eqref{eq:canonical-jets}. By
Lemma~\ref{lem:tilted-jets}, those same sums converge almost surely to
$L_{k,\alpha,u}$. This proves~\eqref{eq:tilted-canonical-jets}.
The equivalence of measures is used here only for almost sure identities.
The $L^2$ assertion for exterior statistics was established separately
under $\Prob_{\alpha,u}$ in Lemma~\ref{lem:tilted-jets}.
\end{proof}

\subsection{The infinite joint transform}\label{subsec:infinite-transform}

Recall that $F_\alpha=U A_\alpha(\Xi)$, where $\Xi\sim H_\alpha$ and
$U$ is an independent Haar phase. Both $|F_\alpha(0)|$ and its normalized
logarithmic jets are measurable functions of $\Xi$. This makes the
joint transform a consequence of the preceding change of measure.

\begin{proposition}\label{prop:infinite-transform}
The identity~\eqref{eq:main-transform} holds for $F_\alpha$ for every
real $u>-1$, integer $q$, and finite vector $\mathbf t\in\C^K$.
\end{proposition}
\begin{proof}
The expectation is well defined, since the Fourier factors have
modulus one and $\E|F_\alpha(0)|^{2u}=M_\alpha(u)<\infty$.
Conditional on $\Xi$, averaging $U^q$ gives zero when $q\ne0$. When
$q=0$, the expectation equals
\[
 \E_{H_\alpha}\left[
 e^{2u(c_\alpha-T_0(\Xi))}
 \exp\left\{i\operatorname{Re}\sum_{k=1}^K\bar t_kL_k(P_\Xi)\right\}
                         \right].
\]
Apply~\eqref{eq:tilted-RN} and~\eqref{eq:tilted-canonical-jets}. The result
is $M_\alpha(u)\Psi_{\alpha,u}^{(K)}(\mathbf t)$, as required.
\end{proof}

The same identities are the limits of the corresponding finite weighted
expectations. Indeed, the joint finite jets converge by
Theorem~\ref{thm:realization}, and the weights have the extra moment used
in~\eqref{eq:weighted-joint-limit}. Their convergence does not require
moments of $L_k$ or an evaluation of a finite determinant.

\subsection{Proof of Theorem~\ref{thm:transform}}\label{subsec:transform-proof}

For completeness, we state the transform uniqueness fact in the form
needed here. It applies to finite candidate vectors whether or not they
are already known to arise from analytic functions.

\begin{lemma}\label{lem:bilateral-uniqueness}
Let $(\ell,V,Y)$ have a probability law on
$\R\times\T\times\R^d$, and suppose
$\E e^{a|\ell|}<\infty$ for some $a>0$. Its law is determined by
\[
 \E[e^{u\ell}V^q e^{it\cdot Y}],
 \qquad |u|<a,\quad q\in\Z,\quad t\in\R^d.
\]
In particular, equality of these transforms on any real neighborhood of
zero implies equality of laws, when both laws have such an exponential
moment.
\end{lemma}
\begin{proof}
For fixed $q,t$, the transform is analytic in the strip
$|\operatorname{Re}u|<a$. On a closed substrip, its derivatives are
dominated by integrable multiples of $e^{a|\ell|}$, using the positive
distance to the edge of the strip. For two laws, use the smaller of
their exponential-moment parameters. Equality on a real neighborhood
of zero and the identity theorem give equality throughout this strip,
hence at all imaginary arguments $u=is$.

For each $q$, define a finite complex measure on $\R^{1+d}$ by weighting
the joint law with $V^q$ and projecting onto $(\ell,Y)$. The equalities at
$u=is$ identify its Fourier transform for every $(s,t)$, and hence identify
this measure. For every Borel set in $\R^{1+d}$, the resulting measures
in the $V$ coordinate therefore have identical integrals against all
trigonometric polynomials. Such polynomials are dense in $C(\T)$.
Equality on rectangles followed by the monotone class theorem identifies
the full joint laws.
\end{proof}

\begin{proof}[Proof of Theorem~\ref{thm:transform}]
Existence and the full real domain $u>-1$ follow from
Proposition~\ref{prop:infinite-transform}. Let $F$ be any candidate in
the theorem, and fix $K$. Write
\[
 \ell=\log|F(0)|^2,\qquad V=\frac{F(0)}{|F(0)|},\qquad
 Y=(\operatorname{Re}L_k(F),\operatorname{Im}L_k(F))_{k=1}^K.
\]
The prescribed transform with $q=0$ and $\mathbf t=0$ gives
$\E e^{u\ell}=M_\alpha(u)$ for real $|u|<\varepsilon$.
Choose $0<a<\varepsilon$. Both $\E e^{a\ell}$ and
$\E e^{-a\ell}$ are finite, so
$\E e^{a|\ell|}<\infty$. The same is true for the corresponding
vector of $F_\alpha$.

Their mixed transforms agree on $|u|<\varepsilon$ for every $q$ and
$\mathbf t$. As $\mathbf t$ varies over $\C^K$, the real pairing
$\operatorname{Re}\sum\bar t_kL_k$ ranges over all real linear
functionals of $Y$. Lemma~\ref{lem:bilateral-uniqueness} therefore
identifies the joint law of $(\ell,V,Y)$ with that of $F_\alpha$.
The analytic continuation in this argument concerns the transforms of
the two probability laws; no analytic continuation of the externally
defined function $\Psi_{\alpha,u}^{(K)}$ is needed.

Since $F(0)=e^{\ell/2}V$, we have identified every finite basepoint--jet
vector. The recursion~\eqref{eq:taylor-recursion} reconstructs each
finite list of Taylor coefficients from such a vector. All finite Taylor
laws thus agree. Lemma~\ref{lem:analytic-coordinates} says that these
coordinates generate the Borel sigma-field of $\Hol$. Consequently
$F\overset d=F_\alpha$ as random analytic functions.
\end{proof}

\appendix
\section{Invariant holomorphic determinantal kernels}
\label{sec:invariant-kernels}\label{sec:class-statement}

Krishnapur's classification~\cite[Chapter~3, Theorem~3.0.5]{Krishnapur2006}
uses a radial Radon reference measure and the projection convention of
his Definition~3.0.2; see also~\cite[Remark~7]{Krishnapur2009}.
We record its extension to arbitrary positive Radon measures and positive
contractions. The rigidity argument is the same in substance: the first
correlation fixes invariant intensity, and the normalized two-point kernel
modulus determines the analytic kernel. Krishnapur differentiates in group
parameters in equations~(3.0.6)--(3.0.9); below we differentiate at the
diagonal to identify invariant curvature. His decisive correlation
argument does not require radiality. The additional conclusions recorded
here are the global holomorphic change of representation, regularity of
the reference measure, and the thinning parameter permitted by
contractivity. The projection point-process laws remain the family
classified by Krishnapur.

We use the normalization~\eqref{eq:model}, writing $B_\alpha=K_\alpha$:
\begin{equation}\label{eq:class-model}
 d\mu_\alpha(z)=\frac{\alpha}{\pi}(1-|z|^2)^{\alpha-1}\dm(z),
 \qquad B_\alpha(z,w)=(1-z\bar w)^{-\alpha-1},\qquad \alpha>0.
\end{equation}
The power uses the analytic logarithm that vanishes at $z\bar w=0$.
The integral operator with kernel $B_\alpha$ on $L^2(\mu_\alpha)$ is
the orthogonal projection $P_\alpha$ onto $A^2(\mu_\alpha)$.
Denote its determinantal law by $H_\alpha$. For $0<p\le1$, independent
$p$-thinning retains each point with probability $p$, independently conditional
on the configuration. Its law $\Thin_p(H_\alpha)$ has kernel $pB_\alpha$
relative to $\mu_\alpha$.

An \emph{admissible holomorphic determinantal representation} consists of a
positive Radon measure $\mu$ on $\D$ and a finite-valued kernel $K$ on
$\D\times\D$ with these properties:
\begin{enumerate}[label=\textup{(\roman*)}]
\item $K$ is holomorphic in its first variable, antiholomorphic in its
second variable, Hermitian, and positive semidefinite as a kernel;
\item its integral operator $Q$ on $L^2(\mu)$ is a locally trace-class
positive contraction, $0\le Q\le I$;
\item the point process $X$ has factorial moment measures
\begin{equation}\label{eq:class-factmom}
 d\mathcal F_n(z_1,\ldots,z_n)
 =\det[K(z_i,z_j)]_{i,j=1}^n\prod_{i=1}^n d\mu(z_i).
\end{equation}
\end{enumerate}
Thus the diagonal in the first correlation formula is the specified
continuous analytic representative. Integral formulas for $Q$ are understood
initially on compactly supported $L^2$ functions. Local finiteness of the
first and second moment measures follows from continuity of $K$, the Radon
property, and $0\le\det[K(z_i,z_j)]_{i,j=1}^2\le K(z_1,z_1)K(z_2,z_2)$.
We exclude the almost surely empty process.

\begin{proposition}[Classification, including thinning]\label{thm:class-main}\label{prop:invariant-kernels}
Let $X$ have an admissible holomorphic determinantal representation $(K,\mu)$
and satisfy $\mathbb P(X\ne\varnothing)>0$. The following conditions are
equivalent:
\begin{enumerate}[label=\textup{(\alph*)}]
\item $\phi X\overset d=X$ for every $\phi\in\Aut(\D)$;
\item both $\mathcal F_1$ and $\mathcal F_2$ are invariant under the
diagonal actions of $\Aut(\D)$;
\item there are $\alpha>0$, $0<p\le1$, and a nowhere-zero
$b\in\Hol$ such that, as measures and pointwise as kernels,
\begin{equation}\label{eq:class-normalform}
 d\mu=|b|^{-2}d\mu_\alpha,
 \qquad K(z,w)=p\,b(z)\overline{b(w)}B_\alpha(z,w).
\end{equation}
\end{enumerate}
Under these conditions,
\begin{equation}\label{eq:class-law}
 \Law(X)=\Thin_p(H_\alpha),
\end{equation}
and $(\alpha,p)$ is determined uniquely by the point-process law. For a
fixed representation $(K,\mu)$, the function $b$ in
\eqref{eq:class-normalform} is unique up to a constant of modulus one.
Moreover, $Q$ is unitarily equivalent, by multiplication by $1/b$, to
$pP_\alpha$. In particular, $Q$ is a projection if and only if $p=1$.
\end{proposition}

\begin{remark}[Meaning of the classification]\label{rem:class-scope}
This theorem concerns Hermitian positive-contraction DPPs admitting a
holomorphic kernel on the whole disk. It does not classify arbitrary invariant
DPPs, non-Hermitian representations, or all random analytic functions whose
zeros are invariant. It assumes neither a radial reference measure nor a
M\"obius transformation rule for $K$ or $\mu$ separately. Those features of
the representation are resolved by the proof.
\end{remark}

\subsection*{Proof of the classification}\label{sec:class-proof}

We use the invariant area measure
\begin{equation}\label{eq:class-area}
 dA_{\mathrm h}(z)=\frac{1}{\pi(1-|z|^2)^2}\dm(z).
\end{equation}
The transitive action of $\Aut(\D)$ identifies the disk with the homogeneous
space $\Aut(\D)/\mathbb T$. Uniqueness of invariant measure on this homogeneous
space implies that every nonzero invariant Radon measure on the disk is
$\lambda A_{\mathrm h}$ for a unique $0<\lambda<\infty$.

\begin{lemma}[The first intensity removes singular reference measures]
\label{lem:class-reference}
Under the assumptions of Proposition~\ref{thm:class-main}, invariance of
$\mathcal F_1$ alone implies
\begin{equation}\label{eq:class-intensity}
 K(z,z)\,d\mu(z)=\lambda\,dA_{\mathrm h}(z),\qquad \lambda>0.
\end{equation}
Furthermore, $K(z,z)>0$ everywhere and
\begin{equation}\label{eq:class-density}
 d\mu(z)=\frac{\lambda}{\pi(1-|z|^2)^2K(z,z)}\dm(z).
\end{equation}
In particular, $\mu$ has full support and a strictly positive smooth density.
\end{lemma}

\begin{proof}
The first intensity is a Radon measure. It is nonzero: if it vanished on
every member of a countable compact exhaustion, the process would be empty
almost surely. Its invariance proves \eqref{eq:class-intensity}.

Let $S=\{a:K(a,a)=0\}$. Kernel positivity gives
$|K(z,w)|^2\le K(z,z)K(w,w)$, so $K(a,w)=0$ for all $w$ whenever $a\in S$.
Choose $w_0$ such that $K(\cdot,w_0)$ is not identically zero. Then $S$ is
contained in the discrete zero set of this holomorphic function.
Suppose $a\in S$. The identities $K(a,w)=K(z,a)=0$ and the local
holomorphic-antiholomorphic power series show that
\[
 K(a+\zeta,a+\eta)=\sum_{j,k\ge1}c_{jk}\zeta^j\bar\eta^k.
\]
Consequently $K(z,z)\le C|z-a|^2$ on a sufficiently small disk about $a$.
Shrink the disk so that it contains no other point of $S$ and has compact
closure in $\D$. On its punctured version, \eqref{eq:class-intensity} permits
division by $K(z,z)$ and gives
\[
 \mu(D(a,r))\ \ge\ c\int_{0<|z-a|<r}|z-a|^{-2}\dm(z)=\infty.
\]
This contradicts local finiteness of $\mu$. Hence $S=\varnothing$.
Division in \eqref{eq:class-intensity} now proves \eqref{eq:class-density}.
\end{proof}

\begin{lemma}
\label{lem:class-curvature}
If both $\mathcal F_1$ and $\mathcal F_2$ are invariant, then there are
$s\ge0$ and a nowhere-zero $h\in\Hol$ such that
\begin{equation}\label{eq:class-hform}
 K(z,w)=h(z)\overline{h(w)}(1-z\bar w)^{-s}.
\end{equation}
\end{lemma}

\begin{proof}
By Lemma~\ref{lem:class-reference}, set
\[
 R(z,w)=\frac{|K(z,w)|^2}{K(z,z)K(w,w)}.
\]
The Radon--Nikodym derivative of $\mathcal F_2$ relative to
$\mathcal F_1\otimes\mathcal F_1$ is $1-R$. Invariance of these measures
gives, for each fixed $\phi\in\Aut(\D)$,
\begin{equation}\label{eq:class-R-invariance}
 R(\phi z,\phi w)=R(z,w).
\end{equation}
Initially this is an almost-everywhere identity. Both sides are continuous,
and the product intensity has full support, so it holds everywhere.

Put $k(z)=K(z,z)$ and $\kappa(z)=\partial_z\partial_{\bar z}\log k(z)$.
For fixed $z$, the kernel $K(z,w)$ is nonzero for $w$ near $z$. Its logarithmic
modulus is harmonic in that neighborhood. Therefore
\begin{equation}\label{eq:class-diagonal-derivative}
 -\left.\partial_w\partial_{\bar w}\log R(z,w)\right|_{w=z}
 =\kappa(z).
\end{equation}
Since $0\le R(z,w)\le1$ and $R(z,z)=1$, this also proves $\kappa(z)\ge0$.
Differentiate \eqref{eq:class-R-invariance} near its diagonal to obtain
\begin{equation}\label{eq:class-metric-invariance}
 \kappa(\phi z)|\phi'(z)|^2=\kappa(z).
\end{equation}
Choosing an automorphism that takes $z$ to zero yields
\begin{equation}\label{eq:class-metric}
 \kappa(z)=\frac{s}{(1-|z|^2)^2},\qquad s=\kappa(0)\ge0.
\end{equation}

The remaining curvature--gauge--polarization step has a direct
analogue in Sodin's Calabi-rigidity argument for Gaussian analytic
functions~\cite[Theorem~2]{Sodin2000}.
Because $\partial\bar\partial[-\log(1-|z|^2)]=(1-|z|^2)^{-2}$,
the real function
$u(z)=\log k(z)+s\log(1-|z|^2)$ is harmonic. Since $\D$ is simply connected, there is a holomorphic $g$ with $2\operatorname{Re}g=u$. Set $h=e^g$.
Then
\[
 \frac{K(z,z)}{|h(z)|^2}=(1-|z|^2)^{-s}.
\]
Both $K(z,w)/(h(z)\overline{h(w)})$ and $(1-z\bar w)^{-s}$ are
holomorphic-antiholomorphic and have this diagonal. Such a kernel is
determined by its diagonal: near any diagonal point its mixed power-series
coefficients are obtained by differentiating the diagonal in $z$ and $\bar z$;
uniqueness then extends throughout $\D\times\D$. This proves
\eqref{eq:class-hform}. No global logarithm of $K(z,w)$ was assumed.
\end{proof}

\begin{lemma}
\label{lem:class-column}
Let $L$ be a continuous Hermitian kernel on $\D\times\D$, relative to a
positive Radon measure $\eta$ of full support, and suppose its integral
operator $T$ is a positive contraction on $L^2(\eta)$. Then, for every $a\in\D$,
\begin{equation}\label{eq:class-column-bound}
 \int_{\D}|L(w,a)|^2\,d\eta(w)\le L(a,a).
\end{equation}
\end{lemma}

\begin{proof}
Let $\epsilon_j\downarrow0$ with $\overline{D(a,\epsilon_j)}\Subset\D$ and
put $f_j=\ind_{D(a,\epsilon_j)}/\eta(D(a,\epsilon_j))$.
These functions belong to $L^2(\eta)$. Since $T^2\le T$,
\[
 \norm{Tf_j}_2^2\le\ip{Tf_j}{f_j}
 =\frac{\displaystyle\int_{D(a,\epsilon_j)^2}L(x,y)
                      \,d\eta(x)\,d\eta(y)}
        {\eta(D(a,\epsilon_j))^2}
 \longrightarrow L(a,a).
\]
Use the integral representative of $Tf_j$. Outside one null set for this
countable sequence, continuity gives $Tf_j(w)\to L(w,a)$. Fatou's lemma
proves \eqref{eq:class-column-bound}. This establishes the pointwise
diagonal inequality without substituting a Dirac mass into an $L^2$ operator.
\end{proof}

\begin{proof}[Proof of Proposition~\ref{thm:class-main}]
The implication (a)$\Rightarrow$(b) is immediate. Assume (b), and use
Lemmas~\ref{lem:class-reference} and \ref{lem:class-curvature}.
Multiplication $Uf=f/h$ is a unitary map from $L^2(\mu)$ onto
$L^2(\widetilde\mu)$, where $d\widetilde\mu=|h|^2d\mu$. The transformed
operator $\widetilde Q=UQU^{-1}$ has kernel and measure
\begin{equation}\label{eq:class-transformed}
 \widetilde K(z,w)=(1-z\bar w)^{-s},\qquad
 d\widetilde\mu(z)=\frac{\lambda}{\pi}(1-|z|^2)^{s-2}\dm(z).
\end{equation}
It remains a positive contraction. Applying Lemma~\ref{lem:class-column}
at $a=0$, where $\widetilde K(w,0)=1$, gives
\begin{equation}\label{eq:class-spectral-bound}
 1\ge\widetilde\mu(\D)
 =2\lambda\int_0^1r(1-r^2)^{s-2}\,dr
 =\lambda\int_0^1t^{s-2}\,dt.
\end{equation}
The integral is finite exactly when $s>1$, in which case it equals
$\lambda/(s-1)$. Thus
\[
 \alpha=s-1>0,\qquad p=\frac{\lambda}{\alpha}\in(0,1].
\]
Equation~\eqref{eq:class-transformed} becomes
$\widetilde K=B_\alpha$, $\widetilde\mu=p\mu_\alpha$.
Set $b=h/\sqrt p$. Then \eqref{eq:class-normalform} holds, proving (c).

For completeness, the monomials satisfy
\[
 \norm{z^n}_{L^2(\mu_\alpha)}^2
 =\alpha B(n+1,\alpha)
 =\frac{\Gamma(\alpha+1)n!}{\Gamma(n+\alpha+1)}.
\]
They form a complete orthogonal system in $A^2(\mu_\alpha)$: angular
integration of an analytic Taylor series gives its norm as the sum of the
squared coefficients with these weights, and its partial sums converge in
that norm. Summing the normalized monomial kernels gives $B_\alpha$.
Multiplication $Vf=f/b$ maps $L^2(\mu)$ unitarily onto $L^2(\mu_\alpha)$,
and direct substitution in the integral gives
\begin{equation}\label{eq:class-unitary}
 VQV^{-1}=pP_\alpha.
\end{equation}
Since $P_\alpha\ne0$, the original operator is a projection exactly when
$p^2=p$, hence exactly when $p=1$.

Under (c), substitution into \eqref{eq:class-factmom} cancels all the
factors $b$ and shows that the factorial moment measures are those of the
DPP with kernel $pB_\alpha$ relative to $\mu_\alpha$. The standard uniqueness
of the DPP associated with a locally trace-class positive contraction
identifies its law. Equivalently, independent thinning multiplies the
$n$th factorial moment measure by $p^n$, which is exactly the same kernel
change. This proves \eqref{eq:class-law}.

The law $H_\alpha$ is invariant under disk automorphisms. One can verify
this directly from its kernel: for
$\phi(z)=e^{i\theta}(a-z)/(1-\bar a z)$ and
\[
 c(z)=(1-|a|^2)^{-(\alpha+1)/2}(1-\bar a z)^{\alpha+1},
\]
the kernel satisfies
$B_\alpha(\phi z,\phi w)=c(z)\overline{c(w)}B_\alpha(z,w)$,
whereas the density of $d\mu_\alpha(\phi z)$ after change of variables is
$|c(z)|^{-2}d\mu_\alpha(z)$. These factors cancel in every correlation
measure. The power in $c$ uses the analytic logarithm at zero; the identity
follows first at zero and then analytically, so the noninteger power causes
no branch ambiguity. Independent thinning commutes with deterministic
point transformations. This proves (c)$\Rightarrow$(a), including the
converse existence assertion for every $\alpha>0$ and $0<p\le1$.

Finally, the intensity of $\Thin_p(H_\alpha)$ is $p\alpha A_{\mathrm h}$,
and its normalized pair correlation is
\begin{equation}\label{eq:class-pair}
 \frac{d\mathcal F_2}{d(\mathcal F_1\otimes\mathcal F_1)}(z,w)
 =1-(1-\delta(z,w)^2)^{\alpha+1},\qquad
 \delta(z,w)=\left|\frac{z-w}{1-z\bar w}\right|.
\end{equation}
The value of this function at any $z\ne w$ determines $\alpha$, and the
intensity then determines $p$. The continuous versions make this conclusion
independent of null-set choices. For fixed $(K,\mu)$, the quotient of two
possible functions $b$ has constant modulus one by the kernel diagonal,
and is therefore a constant of modulus one. All uniqueness assertions follow.
\end{proof}

\begin{corollary}[Projection case without radiality]
\label{cor:class-projection}
A nonempty invariant DPP represented by an admissible holomorphic projection
kernel relative to an arbitrary positive Radon measure on $\D$ has law
$H_\alpha$ for a unique $\alpha>0$.
\end{corollary}

\begin{corollary}
\label{cor:class-two-moments}
Within the admissible holomorphic class, invariance of the first two
factorial moment measures implies invariance of the full law. These two
measures then determine that law.
\end{corollary}

\begin{remark}[Why radiality is a condition on a representation]
For example,
\[
 d\mu(z)=e^{-2\operatorname{Re}z}\,d\mu_\alpha(z),\qquad
 K(z,w)=p e^{z+\bar w}B_\alpha(z,w)
\]
is a nonradial representation of $\Thin_p(H_\alpha)$. Formula
\eqref{eq:class-normalform} proves that every representation covered by the
theorem differs from the radial one in this prescribed way. Removal of
radiality enlarges the permitted representations; it does not introduce
additional projection point-process laws.
\end{remark}

\section{A direct Green-average normalization argument}
\label{sec:direct-normalization}

For the normalized conclusion in Theorem~\ref{thm:geometric}(iii),
independence of the scale can be replaced by the following averaging
argument. Suppose $F$ has divisor law $H_\alpha$, the covariance in
\eqref{eq:main-covariance}, and the two logarithmic moments in
\eqref{eq:geometric-normalization}. The harmonic comparison in
Section~\ref{subsec:independent-scale-proof}, up to
\eqref{eq:random-scale}, gives a spatial constant $B\in L^2$.
At this point $B$ may depend on the zeros. We show directly that $B=0$.

For each fixed $z$, covariance implies
\[
 \log|F(z)|-I_\alpha(z)\overset d=\log|F(0)|.
\]
The normalization~\eqref{eq:geometric-normalization} identifies its
first two moments with those of
$\log|F_\alpha(0)|=c_\alpha-T_0(\Xi)$.
By Lemma~\ref{lem:green-transport}, $T_z$ has the same centered law as
$T_0$. On the other hand,~\eqref{eq:canonical-modulus}
and~\eqref{eq:random-scale} give
\[
 \log|F(z)|-I_\alpha(z)=c_\alpha-T_z(\Xi)+B.
\]
Comparing the first moments yields $\E B=0$. Comparing the second moments
then gives
\begin{equation}\label{eq:scale-identity}
 \E B^2=2\E[B T_z(\Xi)]\qquad(z\in\D\text{ deterministic}).
\end{equation}
The correlation on the right must be retained.

Average the Green symbol in its center. The angular logarithmic mean
formula gives, for $0<r<1$,
\begin{equation}\label{eq:circle-symbol}
 \bar q_r(w):=\frac1{2\pi}\int_0^{2\pi}q_{re^{it}}(w)\,dt
             =-\log\max\{r,|w|\}.
\end{equation}
Except for an irrelevant value at $w=0$,
$0\le\bar q_r(w)\le q_0(w)$, and $\bar q_r(w)\to0$ as $r\uparrow1$.
Since $q_0\in L^2(\nu_\alpha)$, dominated convergence yields
\begin{equation}\label{eq:circle-symbol-small}
 \norm{\bar q_r}_{L^2(\nu_\alpha)}\longrightarrow0.
\end{equation}
Lemma~\ref{lem:green-estimates} makes $t\mapsto q_{re^{it}}$ continuous
in $L^2(\nu_\alpha)$, so its Bochner integral is well defined.
The continuous linear map $\X_\Xi$ commutes with that integral. Thus
\begin{equation}\label{eq:circle-statistic}
 \bar T_r:=\frac1{2\pi}\int_0^{2\pi}T_{re^{it}}(\Xi)dt
             =\X_\Xi(\bar q_r)\longrightarrow0\quad\text{in }L^2.
\end{equation}
The integral of the fixed canonical version agrees with this $L^2$
integral: the versions agree for every deterministic $t$, and joint
Borel measurability and Fubini give their equality for almost every $t$
on almost every path. Their square-integrability on the circle follows
from the uniform Green bound.

We  now integrate~\eqref{eq:scale-identity}; the mixed terms are
integrable by Cauchy--Schwarz. We obtain
\[
 0\le\E B^2=2\E(B\bar T_r)
          \le2\norm B_2\norm{\bar T_r}_2\longrightarrow0.
\]
Consequently $B=0$. This conclusion allows arbitrary dependence between
the scale and the zero configuration.

The candidate is now $U A_\alpha(\Xi)$. Lemma~\ref{lem:forced-phase}
makes $U$ Haar and independent of $\Xi$, so the sampling identity
\eqref{eq:sampling} proves the normalized uniqueness conclusion.

\section*{Acknowledgments}

\noindent 
\textit{Funding.}
X. F. was supported by the National Science and Technology
Council of Taiwan (Grant No.~114-2115-M-A49-003-MY3).
S. H. was supported by the National Natural Science Foundation
of China (Grant No.~12371133).
Q. Z. was supported by the National Natural Science Foundation
of China (Grant No.~12501162) and the Natural Science Foundation of Jiangsu
Province (Grant No.~BK20250832).

\bigskip

\noindent 
\textit{AI Statement.}
The authors used artificial-intelligence tools to assist with exposition,
\LaTeX\ preparation, and the development and checking of some local
arguments. Responsibility for the mathematical content rests with the
authors.

\end{document}